\documentclass[11pt]{preprint}

\usepackage{amsmath, amsthm, amssymb,pdftricks}
\usepackage{mhequ}
\usepackage{mhsymb}
\usepackage{cite}
\usepackage{times}

\newtheorem{theorem}{Theorem}[section]
\newtheorem{corollary}[theorem]{Corollary}
\newtheorem{assumption}[theorem]{Assumption}
\newtheorem{lemma}[theorem]{Lemma}
\newtheorem{proposition}[theorem]{Proposition}

\newtheorem*{claim}{Claim}

\theoremstyle{remark}
\newtheorem{remark}[theorem]{Remark}

\def\x{{\boldsymbol{x}}}
\def\y{{\boldsymbol{y}}}
\def\Y_#1{\boldsymbol{Y}_{\!#1}}
\def\tX{\tilde{\boldsymbol{X}}}
\def\W{\boldsymbol{W}}
\def\E{\mathbb{E}}
\def\P{\mathbb{P}}
\def\R{\mathbb{R}}
\def\TV{\mathrm{TV}}

\definecolor{darkred}{rgb}{0.9,0.1,0.1}

\begin{document}

\title{
Non-asymptotic mixing
of the MALA algorithm
}

\author{N. Bou-Rabee\inst{1}, M. Hairer\inst{1,2}, E. Vanden-Eijnden\inst{1}}
\institute{\scriptsize Courant Institute of Mathematical Sciences, New York University, 
251 Mercer Street, New York, NY 10012-1185\\
 \email{nawaf@cims.nyu.edu, mhairer@cims.nyu.edu, eve2@cims.nyu.edu}
\and
Mathematics Institute, The University of Warwick, Coventry CV4 7AL, U.K.\\
 \email{M.Hairer@Warwick.ac.uk}}

\maketitle

\begin{abstract}
  The Metropolis-Adjusted Langevin Algorithm (MALA), originally
  introduced to sample exactly the invariant measure of certain
  stochastic differential equations (SDE) on infinitely long time
  intervals, can also be used to approximate pathwise the solution of
  these SDEs on finite time intervals. However, when applied to an SDE
  with a nonglobally Lipschitz drift coefficient, the algorithm may
  not have a spectral gap even when the SDE does. This paper
  reconciles MALA's lack of a spectral gap with its ergodicity to the
  invariant measure of the SDE and finite time accuracy.  In
  particular, the paper shows that its convergence to equilibrium
  happens at exponential rate up to terms exponentially small in
  time-stepsize.  This quantification relies on MALA's ability to
  exactly preserve the SDE's invariant measure and accurately
  represent the SDE's transition probability on finite time intervals.
  \\[0.5em]
\end{abstract}

%%%%%%%%%%%%%%%%%%%%%%%%%%%%%%%%%%%%%%%%%%%%%
%%%%%%%%%%%%%%%%%%%%%%%%%%%%%%%%%%%%%%%%%%%%%

\section{Introduction}

The Metropolis-Adjusted Langevin Algorithm (MALA), originally proposed
by Roberts and Tweedie \cite{RoTw1996A, RoTw1996B}, is a technique to
sample exactly complex, high-dimensional probability distributions.
MALA fits the general framework of the Metropolis-Hastings method
\cite{MeRoRoTeTe1953,Ha1970} and can be viewed as a special case of
smart and hybrid Monte-Carlo algorithms \cite{RoDoFr1978,
  DuKePeRo1987}. The main idea of MALA is to obtain the proposal moves
from the forward Euler discretization of an SDE whose invariant
measure is the target distribution one seeks to sample.  Besides being
ergodic with respect to this invariant measure by construction, it was
shown recently that MALA also captures the dynamical behavior of the
solutions to the SDE \cite{BoVa2010A}. Therefore MALA has the nice
feature that it can be used to estimate finite time dynamical
properties along infinitely long trajectories of ergodic SDEs.

Still, one issue with MALA is its theoretical rate of convergence, see
for example \cite{RoTw1996B, CaWaGuMySe2008}.  When applied to
measures with tails that are lighter than Gaussian, it is known that
MALA does not exhibit a geometric rate of convergence to equilibrium
even though the exact solution to the SDE does. The main reason is
that the proposal moves generated by forward Euler are not globally
stable. Indeed for any time-stepsize one can find an energy value
above which the drift in forward Euler gives proposed moves that
increase the energy, in contrast to the exact drift in the SDE which
always centers the solution towards lower energy values.  Since higher
energy values have a lower equilibrium probability weight, these
proposed moves are typically rejected.  While these rejections ensure
that MALA is ergodic, at high energy values they prevent MALA from
having a spectral gap.

The question we investigate in this paper is how severe this problem
is in practical applications. Above we have argued that the main cause
of the lack of geometric convergence is the behavior of the chain at
high energy values. Since the chain is unlikely to reach such high
energy values over finite time horizons, one does not expect their
influence to be significant.  In practice, it is the behavior of MALA
on finite but very long times that is of interest, since this behavior
is what one would experience when running the algorithm on a computer.
The goal of this article is to quantify the non-asymptotic behavior of
MALA.

The main result of this paper states that the convergence of MALA to
its equilibrium distribution happens at exponential rate up to terms
exponentially small in time-stepsize. This can be formulated in the
following way, and will later be reformulated rigorously as
Theorem~\ref{theo:main}:

\begin{claim}
  Let $P_h^n$ denote the $n$-step transition probability of MALA and
  $\mu$ its equilibrium measure.  Set $P = P_h^{\lfloor 1/h \rfloor}$.
  Under natural assumptions on the target distribution $\mu(d
  \boldsymbol{x}) = Z^{-1} \exp(-U( \boldsymbol{x}))\,d\boldsymbol{x}$ (see
  Assumption~\ref{sa}), for $h$ small enough and for all
  $\boldsymbol{x} \in \mathbb{R}^n$ satisfying $U( \boldsymbol{x}) <
  E_0$ there exist positive constants $\rho \in (0,1)$, 
  $C_1(E_0)$ and $C_2$ independent of $h$ such that 
  the bound
\begin{equation} \label{MALAconvergencerate}
\| P^k( \boldsymbol{x}, \cdot) - \mu \|_{\TV} \le  
C_1(E_0) \bigl(\rho^k + e^{-C_2/ h^{1/4}} \bigr)\;, 
\end{equation}
holds for all $k \in \mathbb{N}$.
\end{claim}

Observe from \eqref{MALAconvergencerate} that the distance of MALA to
equilibrium is bounded by the sum of two terms. The first term
converges to $0$ exponentially fast and essentially gives the speed of
convergence to equilibrium for the exact solution to the underlying
SDE.  The second term on the other hand remains bounded away from $0$
as $k \to \infty$.  This term arises from the lack of a spectral gap
in MALA, but its important feature is that it is exponentially small
in $h$.  Therefore, its importance will be negligible in applications
for most practical purposes.

The crux of the proof is the demonstration that MALA inherits some of
the convergence properties of the solution to the underlying SDE up to
exponentially small terms.  This proof relies on finite time accuracy
of MALA, ergodicity of MALA with respect to the exact equilibrium
measure of the SDE, and an application of Harris' theorem.  In fact,
if MALA did not exactly preserve the equilibrium measure of the SDE,
the second term in \eqref{MALAconvergencerate} would not be
exponentially small in the time-stepsize. For example, if MALA was
replaced simply by the uncorrected Euler approximations to the SDE,
then one would expect the size of the error term to be $\CO(h)$.

The estimate \eqref{MALAconvergencerate} does not imply that MALA does
not converge to the equilibrium of the SDE.  In fact, it is known
\cite{RoTw1996B} that the TV distance between MALA and the equilibrium
measure vanishes in the limit as $k \to \infty$.  However, this
asymptotic property provides no insight on the nonasymptotic behavior
of MALA which is the main focus of this paper. In fact, even though
the upper bound in \eqref{MALAconvergencerate} does not converge to
zero in the limit $k \to \infty$, it is the sharpest known bound on
finite time intervals.

The power $1/4$ in the exponentially small term in
\eqref{MALAconvergencerate} is due to the second-order weak accuracy
of the proposal moves generated by the forward Euler scheme, and the
conditions we impose on the potential energy.  In particular, it can
be traced back to the appearance of the factor $U^4(\x)$ appearing in
the statement of Lemma~\ref{MALAphiaccuracy}.  Under the assumptions made in this
paper, this power is sharp.

At the technical level, the main novelty of the proof of our result is twofold.
First, we prove finite-time accuracy of MALA in the total variation 
norm in our setting. While accuracy in total variation of the forward Euler algorithm is known 
\cite{BaTa1995}, it is essential for our analysis to cover situations where the drift of the underlying
SDE is  not globally Lipschitz continuous. Furthermore, we need to keep track of the dependency of the
error estimates with respect to the initial condition. The main idea for this result is to first obtain an error estimate in
some weaker Wasserstein distance, and then to strengthen this into a total variation estimate by making use of the
regularising properties of the one-step transition probabilities of the forward Euler algorithm.
Second, we show that on a very large set, MALA admits a Lyapunov function 
of the type $\Phi(\x) = \exp \bigl(\theta U(\x) \bigr)$ for suitable $\theta > 0$. Since $U$ is allowed to
grow much faster than quadratically at infinity, this Lyapunov function fails to be  integrable with respect to \textit{any} Gaussian 
measure, including of course the transition probabilities of forward Euler. While this leads to 
technical complications, having such a fast-growing Lyapunov function
is a crucial ingredient of our proof, as this is the key to obtaining bounds that are exponentially small
in $h$.

The remainder of this paper is organized as follows.  In Section~\ref{recap},
we will state the main assumptions required for the proof of our main result.  Along the way,
we recall that MALA is ergodic.  In Section~\ref{proof}, the proof of the main result is provided.
This proof relies crucially on comparison with a `patched' MALA algorithm, where the
chain is reflected at the boundaries of a large level set. The accuracy of this
patched algorithm is investigated in Section~\ref{app:MALAtvaccuracy}. Finally, Section~\ref{app:localdriftcondition} 
shows that $\Phi$ is a Lyapunov function for the MALA algorithm (at least on a large domain),
which provides the strong \textit{a priori} bounds
required for our analysis.

\subsection*{Acknowledgements}

{\small
We wish to acknowledge stimulating discussions with Jianfeng Lu, Gareth Roberts, Andrew Stuart, and Jonathan Weare. 
The research of NBR was supported in part by NSF Fellowship DMS-0803095. 
The research of MH was supported by 
an EPSRC Advanced Research Fellowship EP/D071593/1 
and a Royal Society Wolfson Research Merit Award.
The research of EVE was supported in part by NSF grants DMS-0718172 and DMS-0708140, 
and ONR grant N00014-04-1-6046.   
}

\section{A short overview of the MALA algorithm}
\label{recap}

\subsection{Overdamped Langevin equations}

In this paper we focus on overdamped Langevin dynamics on an energy
landscape defined by a potential energy function $U \in \CC^{4}(\mathbb{R}^n, \mathbb{R})$:
\begin{equation}
    d\boldsymbol{Y} = -\nabla U(\boldsymbol{Y}) dt + \sqrt{2 \beta^{-1} } d \boldsymbol{W},~~~
    \boldsymbol{Y}(0) = \boldsymbol{x} \in \mathbb{R}^n \;.
    \label{SDE1}
\end{equation}
Here $\nabla U : \mathbb{R}^n \to \mathbb{R}^n$ denotes the
gradient of the function $U$, $\boldsymbol{W}$ is a standard
$n$-dimensional Wiener process, or Brownian motion, and $\beta>0$ is a
parameter referred to as the inverse temperature.  Under certain
regularity conditions on the potential energy stated in Assumption~\ref{sa} below, 
the solution to \eqref{SDE1} is geometrically ergodic with an invariant probability 
measure $\mu$ that possesses the following density $\pi(\boldsymbol{x})$ with 
respect to Lebesgue measure \cite{Ha1980, RoTw1996B}:
\begin{equation}
  \pi(\boldsymbol{x}) = Z^{-1} \exp(-\beta U(\boldsymbol{x}))
  \label{IM}
\end{equation}
where $Z= \int_{\mathbb{R}^n} \exp(-\beta U(\boldsymbol{x}))d\boldsymbol{x}$.

Before stating assumptions on the potential energy, let us fix some notation.
For a function $G \in \CC^r(\mathbb{R}^n, \mathbb{R})$ and an integer $r>1$, 
let $\nabla G$ and $D^r G$ be the gradient and the $r$th derivative of $G$, respectively.       
Let $| \cdot |$ denote the Euclidean vector norm and $\| \cdot \|$ the Frobenius 
norm.  Let $\mathcal{L}$ denote the generator of \eqref{SDE1} defined for any 
$G \in C^{2}( \mathbb{R}^n, \mathbb{R})$ as
 \begin{align}
 \mathcal{L} G( \boldsymbol{x} ) &\eqdef - \nabla U(\boldsymbol{x}) \cdot \nabla G(\boldsymbol{x}) 
 + \beta^{-1} \Delta G(\boldsymbol{x}) \;. 
 \label{generator}
 \end{align}
For any $t\ge0$, let $Q_t$ denote the transition probabilities of $\boldsymbol{Y}$. We will generally 
make an abuse of notation and use the same symbol for a Markov transition kernel and the associated 
Markov operator.
That is, for any measurable bounded function $\varphi: \mathbb{R}^n \to \mathbb{R}$, we define
$Q_t \varphi: \mathbb{R}^n \to \mathbb{R}$ as \[
( Q_t \varphi) (\boldsymbol{x}) \eqdef 
\int_{\mathbb{R}^n} Q_t(\boldsymbol{x}, d\boldsymbol{y})  \varphi(\boldsymbol{y}) \;.
\] 
Throughout this article, we will make the following assumptions on the potential energy.
Not all of these assumptions will be required for every statement, but we find it
notationally convenient to have a single set of assumptions to refer to.

%.... Structural Assumptions 

\begin{assumption}\label{sa}
The potential energy $U \in \CC^{4}(\mathbb{R}^n, \mathbb{R})$ satisfies the
following.
\begin{itemize}
\item[A)]  
One has $U(\x) \ge 1$ and, for any $C>0$ there exists an $E>0$ such that
\[
U( \boldsymbol{x} ) \ge C ( 1+  | \boldsymbol{x} |^2 ) \;,
\]
for all $U(\boldsymbol{x}) > E$.  
\item[B)] 
There exist constants $c \in (0,\beta)$, $d>0$ and $E>0$ such that 
\begin{equ}[e:boundDDU]
  \Delta U(\boldsymbol{x}) \le c |\nabla U(\boldsymbol{x})|^2 - d U(\boldsymbol{x}) \;,
\end{equ}
for all $\boldsymbol{x} \in \mathbb{R}^n$ satisfying $U(\boldsymbol{x})>E$. 
\item[C)]
The Hessian of $U$ is bounded from below in the sense that there exists $C \ge 0$
such that 
\begin{equ}
D^2U(\x)(\boldsymbol{\eta},\boldsymbol{\eta}) \ge -C | \boldsymbol{\eta} |^2\;,
\end{equ}
uniformly for all $\x, \boldsymbol{\eta} \in \R^n$.
\item[D)] 
There exists a constant $C>0$ such that
the first four derivatives of the potential energy 
$U \in C^{4}(\mathbb{R}^n, \mathbb{R})$ are bounded
by the potential energy itself, that is
\[
\left\| D^4 U( \boldsymbol{x} ) \right\| \vee 
\left\| D^3 U( \boldsymbol{x} ) \right\| \vee 
\left\| D^2 U( \boldsymbol{x} ) \right\| \vee  
\left| \nabla U( \boldsymbol{x} ) \right|  \le C U( \boldsymbol{x} ) \;,
 \]
for all $\boldsymbol{x} \in \mathbb{R}^n$.  Recall, the function $\vee$ returns the argument with the maximum value.
%\item[E)]
%There exist positive constants $\alpha$ and $E_c$ such that for every $E>E_c$ there is a region 
%$S_E \subset \mathbb{R}^n$ satisfying
%\begin{enumerate}
%\item $\{ \boldsymbol{x} ~|~ U(\boldsymbol{x}) < E \} \subset S_E \subset \{ \boldsymbol{x} ~|~ U(\boldsymbol{x}) <  2 E \}$ \;,
%\item $\Omega(S_E, E^{-2}) > \alpha$ \;.
%\end{enumerate}
%Here we have used the notation 
%\begin{equ}
%\Omega(S,r) \eqdef \inf_{\boldsymbol{x} \in S} 
%\sup_{\CC \in \CC_{\boldsymbol{x}}(r)\,:\, \CC\subset S} |\CC| \;,
%\end{equ}
%where $\CC_x(r)$ denotes the set of all closed hypercones of radius $r$ with apex at $\boldsymbol{x}$,
%and $|\CC|$ is the normalised solid angle spanned by $\CC$. In other words, for every closed $A \subset S^{n-1}$, we have
%a hypercone $\CC \in \CC_{\boldsymbol{x}}(r)$ given by $\CC = \{\boldsymbol{x} + cy\,:\, c \in [0,r]\; y \in A\}$,
%and we have $|\CC| = \lambda(A)/\lambda(S^{n-1})$, where $\lambda$ is the $n-1$-dimensional Hausdorff measure.
\end{itemize}
\end{assumption}

\begin{remark}\label{rem:boundTail}
It follows immediately from
Assumption~\ref{sa} (A) above that exists a constant $E_c>0$ such that 
\begin{equ}[e:boundHigh]
\mu\bigl(\{U(\boldsymbol{x}) \ge E\}\bigr) \le e^{-{\beta E \over 2}} \;,
\end{equ}
for all $E > E_c$. Indeed, it suffices to note that 
\begin{equs}
\mu\bigl(\{U(\boldsymbol{x}) \ge E\}\bigr) &= {1\over Z}\int_{U(\boldsymbol{x}) \ge E} e^{- \beta U(\boldsymbol{x})} d\boldsymbol{x} 
\le {e^{-3 \beta E/4}\over Z} \int_{U(\boldsymbol{x}) \ge E} e^{- \beta U(\boldsymbol{x})/4} d\boldsymbol{x} \\
&\le C e^{-3 \beta E/4}  < e^{- \beta E/2}\;,
\end{equs}
where the second to last inequality follows from point (A) above, and the
last inequality holds for $E$ sufficiently large.
\end{remark}

%\begin{remark}
%Point (E) may look strange at first sight, but it is a geometric condition that is typically quite
%easy to verify. All it says is that the shape of the level sets of $U$ does not behave too wildly at
%high energies. This condition will be used in the proof of Lemma~\ref{patchedMALAdrift} below in order to ensure that
%typical Euler steps do not result in too large increases of the energy. Typically, level sets become
%more and more regular at high energies, so that one can usually take $\alpha$ arbitrarily close to ${1\over 2}$.
%\end{remark}
%
\begin{remark}
The only place where we actually use the fact that $U(\x)$ grows like $|\x|^2$ is in
the proof of Lemma~\ref{MALAphiaccuracy} below. On the other hand, the statement of that
approximation result would certainly be true also for potentials that grow slower at $\infty$. However,
such potentials would not be of interest for the present work.  Indeed, if the potential grows slower than $|\x|^2$ and no slower than
$| \x|$, then MALA can be shown to be exponentially ergodic, so that the results in this 
article would be superfluous.  If the potential grows slower than $|\x|$, then MALA will not be exponentially
ergodic because the true solution of the SDE will not be either.
\end{remark}

\begin{remark} \label{oslremark}
Assumption~\ref{sa} (C) is equivalent to the existence of $C>0$ such that $\nabla U$ satisfies the one-sided Lipschitz property
\[
\left\langle - \nabla U( \boldsymbol{x} ) + \nabla U( \boldsymbol{y} ) , 
\boldsymbol{x} - \boldsymbol{y} \right\rangle  
\le  C | \boldsymbol{x} - \boldsymbol{y} |^2,
~~\forall~\boldsymbol{x}, \boldsymbol{y} \in \mathbb{R}^n \;.
\]
\end{remark}

All of these conditions  are satisfied, for example, if $U$ is smooth and 
$U(\boldsymbol{x}) \approx |\boldsymbol{x}|^\alpha$ with $\alpha > 2$ for 
large values of $\boldsymbol{x}$. However, they also allow for potentials that have very asymmetric
growth at infinity, and they even allow for the potential to grow at exponential speed. 
As a consequence of Assumption~\ref{sa}~(B),
one has the following drift condition on the transition probability of the solution.

\begin{lemma}\label{lem:driftGen}
Let $\Theta \colon \R_+ \to \R$ be a $\CC^2$ function such that there exist $u_0 > 0$ and $\alpha  > 0$
such that $\Theta(u) > 0$, $\Theta'(u) > 0$, $u \Theta'(u) > \alpha \Theta(u)$, and $\Theta''(u) \le (\beta - c) \Theta'(u)$ for $u > u_0$. 
(Here, the constant $c$ is the one appearing in \eref{e:boundDDU} above.)
Then, there exist positive constants $K_\Theta$ and $\gamma_\Theta$ such that 
\begin{equ}
\CL \bigl(\Theta \circ U\bigr) \le K_\Theta - \gamma_\Theta \, \bigl(\Theta \circ U\bigr)\;.
\end{equ}
In particular,
\begin{equ}[e:Lyap]
\bigl(Q_t (\Theta \circ U)\bigr) (\boldsymbol{x}) \le e^{-\gamma_\Theta t} \Theta(U(\boldsymbol{x})) 
+ {K_\Theta\over \gamma_\Theta} ( 1 - e^{-\gamma_\Theta t} ) 
\end{equ}
holds for every $t \ge 0$ and for every $\boldsymbol{x} \in \mathbb{R}^n$.
\end{lemma}

\begin{proof}
Using the specific form of $\CL$, it follows that for $U(\x) > u_0$, we have
\begin{equs}
\CL  \bigl(\Theta \circ U\bigr) &= \bigl(\Theta' \circ U\bigr)\,\CL U +  {1\over \beta} \bigl(\Theta'' \circ U\bigr) |\nabla U|^2 \\
&=  \Bigl({\bigl(\Theta'' \circ U \bigr) \over \beta} - \bigl(\Theta' \circ U \bigr)\Bigr)\,|\nabla U|^2 +  {\bigl(\Theta' \circ U\bigr)\over \beta} \Delta U \\
&\le  {1\over \beta} \bigl( (\Theta'' \circ U)  - (\beta - c)(\Theta' \circ U)\bigr)\,|\nabla U|^2 -{d\over \beta}  \bigl(\Theta' \circ U\bigr) U \\
&\le-{d\alpha \over \beta}  \bigl(\Theta \circ U\bigr)\;.
\end{equs}
The result then follows at once from the fact that the condition $u \Theta'(u) > \alpha \Theta(u)$ implies that $\Theta(u) \to \infty$ as 
$u \to \infty$.
\end{proof}

\begin{remark}\label{SDEdriftcondition2}
The condition of Lemma~\ref{lem:driftGen} holds for example for $\Theta(u) = \exp(\theta u)$, provided that 
$\theta < \beta - c$. It also holds for $\Theta(u) = u^\ell$ for every $\ell > 0$ and for $\Theta(u) = u^\ell\exp(\theta u)$
with the same constraints on $\ell$ and $\theta$.
This will be useful in the sequel. Throughout this article, we will write $\Phi(\x) = \exp(\theta U(\x))$ for
some unspecified $\theta < \beta - c$, so that
\begin{equ}[e:LyapPhi]
\CL\Phi \le K - \gamma\Phi\;.
\end{equ}
When the precise value of $\theta$ matters, we will denote the corresponding function by $\Phi_\theta$.
\end{remark}

As a consequence of the ellipticity of the SDE \eqref{SDE1}, one has the 
following minorization condition on the solution's transition probability.

\begin{lemma}\label{SDEminorization}
For every $t>0$ and $E>0$, there exists $\epsilon > 0$ such that 
 \begin{equation} \label{eq:minor}
\| Q_t(\boldsymbol{x}, \cdot) - Q_t(\boldsymbol{y}, \cdot) \|_{\TV} \le 2(1 - \epsilon)\;,
\end{equation}
for all $\boldsymbol{x}, \boldsymbol{y}$ satisfying 
$U(\boldsymbol{x}) \vee U(\boldsymbol{y}) < E$.  
\end{lemma}

\begin{remark}
Here and in the sequel, the total variation distance between two probability measures
is defined as
\begin{equ}
\|\mu - \nu\|_\TV = 2\sup_A |\mu(A) - \nu(A)|\;,
\end{equ}
where the supremum runs over all measurable sets. In particular, the total variation distance
between two probability measures is two if and only if they are mutually singular.
\end{remark}

\begin{proof}[Proof of Lemma~\ref{SDEminorization}]
It follows from the ellipticity of the equations that there exists a 
function $q(t,\boldsymbol{x},\boldsymbol{y})$  smooth in all of its arguments 
(for $t > 0$) such that the transition probabilities are given by  
$Q_t(\boldsymbol{x},d\boldsymbol{y}) = q(t,\boldsymbol{x},\boldsymbol{y})\,d\boldsymbol{y}$. 
Furthermore, $q$ is strictly positive (see, e.g., Lemma 2.2 of \cite{Ta2002}). 
Hence, by the compactness of the set $\{ \boldsymbol{x} \,:\, U(\boldsymbol{x}) < E\}$, one 
can find a probability measure $\eta$ and a constant $\epsilon>0$ such that, \[
Q_t(\boldsymbol{x}, \cdot) > \epsilon \eta(\cdot) 
\]
for any $\boldsymbol{x}$ satisfying $U(\boldsymbol{x})<E$.   This condition implies the
following transition probability $\tilde{Q}_t$ is well-defined: \[
\tilde{Q}_t(\boldsymbol{x}, \cdot) 
= \frac{1}{1-\epsilon} Q_t(\boldsymbol{x}, \cdot) - \frac{\epsilon}{1-\epsilon} \eta(\cdot) 
\]
for any $\boldsymbol{x}$ satisfying $\Phi(\boldsymbol{x})<E$.   Therefore, \[
\| Q_t(\boldsymbol{x}, \cdot) - Q_t(\boldsymbol{y}, \cdot) \|_{\TV} =
(1-\epsilon) \|  \tilde{Q}_t(\boldsymbol{x}, \cdot) - \tilde{Q}_t(\boldsymbol{y}, \cdot) \|_{\TV}
\]
for all $\boldsymbol{x}, \boldsymbol{y}$ satisfying 
$U(\boldsymbol{x}) \vee U(\boldsymbol{y}) < E$.  
Since the TV norm is bounded by $2$, one obtains the desired result.
\end{proof}

Harris' theorem can now be invoked to conclude the transition
probability of the true solution converges at a geometric rate 
to its equilibrium measure.  For the reader's convenience, we state the precise version
used in this article. For a proof, see the monograph \cite{MeTw2009}, or \cite{HaMa2008} for a shorter 
and somewhat more constructive version.
Harris' theorem essentially states that if a Markov chain $\CP$ on an arbitrary (Polish) state space $\CX$
admits a Lyapunov function such that its sublevel sets are `small', then it is exponentially ergodic.
More precisely, Harris' theorem applies to any Markov chain that satisfies the following assumptions:

\begin{assumption}[Drift Condition]
There exists a function $\Phi: \CX \to \mathbb{R}^+$ and constants 
$\gamma \in (0,1)$ and $K \ge 0$, such that the Markov chain $\mathcal{P}$ 
satisfies \begin{equation} \label{drift}
(\mathcal{P} \Phi )( \boldsymbol{x}) \le \gamma \Phi(\boldsymbol{x}) + K\;,
\end{equation}
for all $\boldsymbol{x} \in \CX$.
\label{driftcondition}
\end{assumption} 

\begin{assumption}[Associated `Minorization' Condition]
There exists a constant $\alpha \in (0,1)$ so that 
the Markov chain $\mathcal{P}$ satisfies 
\begin{equation} \label{weakerminor}
\| \mathcal{P}(\boldsymbol{x}, \cdot) - \mathcal{P}(\boldsymbol{y}, \cdot) \|_{\TV} \le 2(1 - \alpha)\;,
\end{equation}
for all $\boldsymbol{x}, \boldsymbol{y} \in \mathbb{R}^n$ with 
$\Phi(\boldsymbol{x}) + \Phi(\boldsymbol{y}) \le 4 K / (1- \gamma)$, where
 $K$ and $\gamma$ are the constants from Assumption~\ref{driftcondition}.
\label{minorization}
\end{assumption}

Note that in this statement, we have normalised the total variation distance between two probability
measures in such a  way that it is equal to $2$ if and only if the measures are mutually singular.
One then has:

\begin{theorem}[Harris' theorem]\label{theo:Harris}
Suppose a Markov chain $\mathcal{P}(\boldsymbol{x}, dy)$ on $\mathbb{R}^n$ satisfies 
Assumptions~\ref{driftcondition}  and~\ref{minorization}.  
Then there exists a unique invariant measure $\mu$ for $\CP$ and there are constants $C>0$
and $\rho < 1$, both depending only on the constants $\gamma$, $K$ and $\alpha$ 
appearing in the assumptions, such that
\[
\|\mathcal{P}^n(\boldsymbol{x},\cdot\,) - \mu\|_\TV \le C\rho^n \Phi(\boldsymbol{x})\;,
\]
for any $\boldsymbol{x} \in \mathbb{R}^n$.  
\label{harris}
\end{theorem}

With this tool at hand, we obtain the following exponential ergodicity result for
the solutions to \eref{SDE1}:

\begin{theorem} \label{SDEgeometricallyergodic}
Let $U$ be a potential function satisfying Assumption~\ref{sa}.
Then, for every $\theta \in (0, \beta-c)$ there exist positive constants 
$\delta \in (0,1)$ and $C$ such that 
\begin{equ}[e:boundTV]
 \| Q_t^k(\boldsymbol{x}, \cdot) - \mu \|_{\TV}  
 \le C \delta^k \exp\bigl(\theta U(\boldsymbol{x})\bigr)
\end{equ}
for all $t>0$ and all $\boldsymbol{x} \in \R^n$.
\end{theorem}

\begin{proof}
According to Remark~\ref{SDEdriftcondition2}, for every $\theta \in (0, \beta-c)$, 
$\exp(\theta U)$ is a Lyapunov function for the Markov chain $Q_t$.  Moreover, by 
Lemma~\ref{SDEminorization} it satisfies a minorization condition on every sublevel 
set of $U$. Hence, Harris' theorem implies that \eref{e:boundTV} holds.
\end{proof}

Next we recall some integration strategies for \eqref{SDE1} and summarize
their properties.  In particular, we discuss to what extent these strategies preserve
the geometric rate of convergence of the true solution.

\subsection{Forward Euler}

Let the time-stepsize $h$ be given, set $t_k = h k $ for
$k \in \mathbb{N}$, and consider the following forward Euler
discretization of~\eqref{SDE1}:
\begin{equation}
    \Tilde{\boldsymbol{X}}_{k+1} =\Tilde{\boldsymbol{X}}_k 
    - h \nabla U(\Tilde{\boldsymbol{X}}_k)
    +  \sqrt{2 \beta^{-1}}  (\boldsymbol{W}(t_{k+1}) - \boldsymbol{W}(t_{k})) \;,~~~
      \Tilde{\boldsymbol{X}}_{0} = \boldsymbol{x} \in \mathbb{R}^n \;.
    \label{ULA}
\end{equation}
Here $ \Tilde{\boldsymbol{X}}_k$ should be viewed as an approximation to
$\boldsymbol{Y}_k \eqdef \boldsymbol{Y}(t_k)$.
The iteration rule~\eqref{ULA} defines a Markov chain that possesses a
transition probability with the following smooth, strictly positive 
transition density:
\begin{equation}
  q_h(\boldsymbol{x}, \boldsymbol{y}) = (4 \pi \beta^{-1} h)^{-n/2} 
  \exp\left(  -  \frac{\left| \boldsymbol{y} - \boldsymbol{x} 
        + h \nabla U(\boldsymbol{x} ) \right|^2}{ 4 \beta^{-1} h} \right)  \;.
  \label{Eulerdensity}
\end{equation}  
Hence, the chain is irreducible with respect to Lebesgue
measure.

If $\nabla U$ is globally Lipschitz and $h$ is small
enough, forward Euler~\eqref{ULA} can be shown to be
exponentially ergodic  with respect to a probability distribution
that is a first-order approximant to the equilibrium distribution of the SDE \eqref{SDE1}.
This property is typically established using a Talay-Tubaro expansion of the 
global weak error of  forward Euler \cite{TaTu1990}.

When $\nabla U$ is nonglobally Lipschitz, forward Euler 
is a transient Markov chain for any $h>0$.  In fact, all moments of forward Euler 
are unbounded on long time-intervals for any initial condition 
$\boldsymbol{x} \in \mathbb{R}^n$.  To be precise for any
integer $\ell \ge 1$ and for any $h>0$
\begin{equation}
  \E^{\boldsymbol{x}} | \Tilde{\boldsymbol{X}}_k |^{ \ell} 
  \to \infty  \quad \text{as} \quad k \to \infty \;,
\label{emunboundedmoments}
\end{equation}
where $\mathbb{E}^{\boldsymbol{x}}$ denotes the expectation 
conditional on $\Tilde{\boldsymbol{X}}_0=\boldsymbol{x}$, 
see e.g.\ \cite{MaStHi2002, Ta2002}.   This instability implies that
an equilibrium trajectory of forward Euler does not sample any 
probability distribution.  As is well known in the literature, 
a Metropolis-Hastings method can stochastically stabilize forward Euler.

\subsection{MALA Algorithm}
\label{sec:defMALA}
%\subsubsection{Ergodicity}

A Metropolis-Hastings method is a Monte-Carlo method for
producing samples from a known probability distribution \cite{MeRoRoTeTe1953,Ha1970}.    
The method generates a Markov chain from a given proposal Markov chain as follows.  
A proposal move is computed according to the proposal chain and
accepted with a probability that ensures the
Metropolized chain is ergodic with respect to the given probability
distribution.  Here we shall focus on the Metropolized forward Euler
integrator defined in terms of the equilibrium density $\pi$~\eqref{IM}
and the transition density $q_h$~\eqref{Eulerdensity}.

Given a time-stepsize $h$ 
and input state $\boldsymbol{X}_k$ the algorithm calculates a proposal move
using the forward Euler updating scheme in~\eqref{ULA}:
\begin{equation}   \label{MALAproposal}
  \boldsymbol{X}^{\star}_{k+1} = \boldsymbol{X}_k - h
  \nabla U(\boldsymbol{X}_k ) + \sqrt{2 \beta^{-1}}
  (\boldsymbol{W}(t_{k+1}) - \boldsymbol{W}(t_{k})) \;,
\end{equation}
and accepts this proposal with a probability 
\begin{equation}
  \label{MALAacceptreject}
  \alpha_h(\boldsymbol{x}, \boldsymbol{y}) =  1 \wedge 
  \frac{q_h(\boldsymbol{y},\boldsymbol{x}) \pi(\boldsymbol{y}) }
  { q_h(\boldsymbol{x},\boldsymbol{y}) \pi(\boldsymbol{x}) }  \;.
\end{equation}
In other words, if $\zeta_k \sim U(0,1)$ is an i.i.d.\ sequence of uniformly distributed random 
variables, the update is defined as:
\begin{equ}[MALA]
  \boldsymbol{X}_{k+1} =
\begin{cases}
  \boldsymbol{X}^{\star}_{k+1}\qquad &
  \text{if}~~\zeta_k<\alpha_h(\boldsymbol{X}_k,
  \boldsymbol{X}^{\star}_{k+1} )\\
  \boldsymbol{X}_k & \text{otherwise}
\end{cases} 
\end{equ}
for $k \in \mathbb{N}$.  To be consistent with the literature, we will refer 
to the Metropolized forward Euler integrator as the Metropolis-adjusted 
Langevin algorithm (MALA) \cite{RoTw1996A}.  We emphasize that MALA 
is a special case of the smart and hybrid Monte-Carlo algorithms which are older and 
more general sampling methods, see \cite{RoDoFr1978, DuKePeRo1987}.   
By construction, MALA preserves the invariant measure $\mu$ of \eqref{SDE1}.  
This implies for any $g: \mathbb{R}^n \to \mathbb{R}$,
\begin{equation} \label{Xpreservesmu}
 \E^{\mu} \left( g( \boldsymbol{X}_k ) \right) 
= \int_{\mathbb{R}^n} g(\boldsymbol{x}) \mu(d \boldsymbol{x}),
~~\forall~~ k \in \mathbb{N} \;.
\end{equation}
Here $\E^{\mu}$ denotes
expectation conditioned on the initial distribution of the integrator 
being the equilibrium distribution of the SDE \eqref{SDE1}:
\[
 \mathbb{E}^{\mu} \left( g( \boldsymbol{X}_k ) \right) = 
\int_{\mathbb{R}^n} 
\mathbb{E}^{\boldsymbol{x}} \left( g( \boldsymbol{X}_k ) \right) \, \mu(d \boldsymbol{x}),~~~
\boldsymbol{X}_0 = \boldsymbol{x} \in \mathbb{R}^n \;.
\]
Moreover, it is quite standard to show that MALA gives rise to an ergodic Markov chain.
Indeed, denoting by $P_h$ the transition probabilities defined by \eref{MALA}, one has

\begin{theorem}[Roberts and Tweedie, \cite{RoTw1996B}]
Let $U$ be a potential satisfying Assumption~\ref{sa}.
For any $h>0$ the $k$-step transition probability of MALA 
converges to $\mu$ in the total variation metric on probability measures, that is
\[
\lim_{k \to \infty} \| P_h^k(\boldsymbol{x}, \cdot) - \mu \|_{\TV} = 0\;,
\]
for all $\boldsymbol{x} \in \mathbb{R}^n$.
\label{MALAergodicity}
\end{theorem}

 If $\nabla U$ is globally Lipschitz and $h$ is 
small enough, MALA is geometrically ergodic 
(see Theorem 4.1 of\cite{RoTw1996B}).   However,  if $\nabla U$ is nonglobally Lipschitz, 
MALA is not geometrically ergodic even though the solution to the SDE is
(see Theorem 4.2 of\cite{RoTw1996B}).  Specifically, one can prove the following.

\begin{theorem}[Roberts and Tweedie, \cite{RoTw1996B}] 
Let $U$ be a potential satisfying Assumption~\ref{sa}. 
If \begin{equation}  \label{RoTw1996B}
\liminf_{|\boldsymbol{x}| \to \infty} \frac{| \nabla U(\boldsymbol{x}) |}{| \boldsymbol{x} |} 
> \frac{2 \beta}{h}
\end{equation}
then MALA operated at time-stepsize $h$ is not geometrically ergodic.
\label{MALAnotgeoergodic}
\end{theorem}

If \eqref{RoTw1996B} holds, the tail of the equilibrium density is no heavier than Gaussian.  
For example, if $U(x) = x^4/4$
then \[
\liminf_{|x| \to \infty} \frac{| \nabla U(x) |}{| x |}   = \infty  \;.
\]
In this case the theorem states MALA is not geometrically ergodic, in contrast
to the true solution of the SDE. The main purpose of this article is to argue that,
up to errors that are \textit{exponentially} small in the time-step size $h$,
the convergence of the transition probabilities of MALA towards equilibrium still
takes place at an exponential rate. The next section gives a precise statement of this result,
as well as an overview of its proof.

%%%%%%%%%%%%%%%%%%%%%%%%%%%%%%%%%%%%%%%%%%%%%
%%%%%%%%%%%%%%%%%%%%%%%%%%%%%%%%%%%%%%%%%%%%%

\section{Main Results}
\label{proof}

We now state and prove the main result of the paper. Throughout this section, $P_h$ will
denote the one-step transition probabilities of the MALA algorithm as defined in Section~\ref{sec:defMALA}
above. We will also use throughout this section the shorthand notation $P = P_h^{\lfloor 1/h \rfloor}$ for 
the evolution of MALA over one unit of `physical time'.

%.....  Main Result of Paper

\begin{theorem}\label{theo:main}
Let $U$ be a potential function satisfying Assumption~\ref{sa}, and let $P$ be as above.
Then, there exists $\bar{\delta} \in (0,1)$
and, for every $E_0>0$, there exist positive 
constants $C_1$, $C_2$, and $h_c(E_0)$ 
such that MALA's distance  to stationarity satisfies
\[
\| P^k( \boldsymbol{x}, \cdot) - \mu \|_{\TV} \le  
C_1  \Phi(\boldsymbol{x}) \bigl(\bar{\delta}^k + e^{-C_2/ h^{1/4}}\bigr) \;,
\]
for all $k \in \mathbb{N}$, all stepsizes $h<h_c$, and all $\boldsymbol{x}$ satisfying 
$U( \boldsymbol{x}) < E_0$.
\end{theorem}

To quantify MALA's distance to stationarity, $\| P^k(\boldsymbol{x}, \cdot ) - \mu \|_{\TV} $, 
we adopt a patching argument.  The point of the 
patching argument is to use compactness to boost a local property of MALA to a 
global property.  The main ingredient of this argument is a version of MALA with reflection
on the boundaries of certain compact sets.

%... introduce `patched MALA'

To introduce this patched version of MALA, set  
$R_h = \{ \boldsymbol{x} : U(\boldsymbol{x}) < E_h \}$,
where  $E_h = E_\star h^{-1/4}$ for a constant $E_\star$ yet to be determined. The
`patched MALA' algorithm  is then defined as a Metropolized version of forward Euler with
a reflecting boundary condition at the boundary of $R_h$.   This boundary condition is 
enforced by setting the target distribution in MALA to be the equilibrium 
distribution $\mu$ conditional on being in $R_h$.   This distribution possesses the following density with respect to 
Lebesgue measure: 
\begin{equation}
\bar{\pi}( \boldsymbol{x} ) = Z_h^{-1}   e^{-\beta U(\boldsymbol{x})} \mathbf{1}_{R_h}(\boldsymbol{x})  \;,
\end{equation}
where $Z_h = \int_{R_h} \exp(-\beta U(\boldsymbol{x}))d\boldsymbol{x}$ and $\mathbf{1}_{R_h}$ is the
indicator function for the set $R_h \in \mathbb{R}^n$.

To be more precise, given a time-stepsize $h$ 
and input state $\bar{\boldsymbol{X}}_k \in R_h$, the algorithm calculates a proposal move
using the forward Euler updating scheme in~\eqref{ULA}:
\begin{equation}   \label{patchedMALAproposal}
  \bar{\boldsymbol{X}}^{\star}_{k+1} = \bar{\boldsymbol{X}}_k - h
  \nabla U(\bar{\boldsymbol{X}}_k ) + \sqrt{2 \beta^{-1}}
  (\boldsymbol{W}(t_{k+1}) - \boldsymbol{W}(t_{k})) \;,
\end{equation}
and accepts this proposal with a probability 
\begin{equation}
  \label{patchedMALAacceptreject}
 \bar{\alpha}_h(\boldsymbol{x}, \boldsymbol{y}) =  \begin{cases}
 1 \wedge 
  \frac{q_h(\boldsymbol{y},\boldsymbol{x}) \pi(\boldsymbol{y}) }
  { q_h(\boldsymbol{x},\boldsymbol{y}) \pi(\boldsymbol{x}) }  & \text{if}~~\boldsymbol{x} \in R_h \\
  0 & \text{otherwise} \;.
  \end{cases}
\end{equation}
In other words, if $\zeta_k \sim U(0,1)$ is an i.i.d.\ sequence of uniformly distributed random 
variables, the update is defined as:
\begin{equ}[Patched MALA]
  \bar{\boldsymbol{X}}_{k+1} =
\begin{cases}
  \bar{\boldsymbol{X}}^{\star}_{k+1}\qquad &
  \text{if}~~\zeta_k<\bar{\alpha}_h(\bar{\boldsymbol{X}}_k,
  \bar{\boldsymbol{X}}^{\star}_{k+1} )\\
  \bar{\boldsymbol{X}}_k & \text{otherwise}
\end{cases} 
\end{equ}
for $k \in \mathbb{N}$.  
We stress that patched MALA always remains in $R_h$
since it rejects all moves to $R_h^c$.  
Let $\bar{P}_h$ denote the transition probability of patched MALA.
Let $\bar{\mu}$ denote the invariant measure of $\bar{P}_h$ with density 
$\bar{\pi}$.   The invariant measures of $\bar{P}_h$ and $P_h$ 
are related by: 
\begin{equ}[e:defmubar]
\bar{\mu}(A) = \frac{ \mu(A \cap R_h) }{\mu(R_h)}
\end{equ}
for all measureable sets $A$.
Set $\bar{P} = \bar{P}_h^{\lfloor 1/h \rfloor}$.  With this notation we are
ready to prove Theorem~\ref{theo:main}.

\begin{proof}[Proof of Main Result]  %  of Main Result of Paper ................................................................................%

This proof relies on Lemmas~\ref{1term} and \ref{2term} provided below.
Using the triangle inequality, we bound the distance of $P^k$ to stationarity by 
\begin{equs}
 \| P^k(\boldsymbol{x}, \cdot) - \mu \|_{\TV}  \nonumber &\le \| P^k(\boldsymbol{x}, \cdot)- \bar{P}^k(\boldsymbol{x}, \cdot) \|_{\TV}  
+ \| \bar{P}^k(\boldsymbol{x}, \cdot) - \bar{\mu} \|_{\TV}  
+ \| \bar{\mu} - \mu \|_{\TV}  \\
& \eqdef I_1 + I_2 + I_3 \;.
\label{patching}
\end{equs}
We now bound all three terms separately.

%... bound I_1 using coupling between MALA & `patched MALA'

Lemma~\ref{1term} bounds $I_1$ in \eqref{patching} using a coupling between MALA and 
patched MALA, and the coupling characterization of the total variation distance.
The lemma states for every $E_0>0$ there exist positive constants
$\tilde{C}_1$ and $h_c$ such that
\begin{equation} \label{1termbound}
I_1 = \| P^k(\boldsymbol{x}, \cdot)- \bar{P}^k(\boldsymbol{x}, \cdot) \|_{\TV}  \le     
\tilde{C}_1 \Phi(\boldsymbol{x}) e^{-\beta E_h} k   \;,
\end{equation}
for all $h<h_c$ and every $\boldsymbol{x}$ satisfying $U( \boldsymbol{x}) < E_0$.

%... bound I_2 using Harris' theorem

Lemma~\ref{2term} bounds $I_2$ in \eqref{patching} 
by using Harris' theorem, Theorem~\ref{theo:Harris}. This lemma relies on a drift and minorization condition 
for patched MALA.  The lemma states that patched MALA is exponentially 
ergodic, that is, for every $\bar{\delta} \in (\delta,1)$ and $E_0>0$, 
there exist positive constants $C_3$ and $h_c$ such that
\begin{equation} \label{2termbound}
I_2 = \| \bar{P}^k(\boldsymbol{x}, \cdot) - \mu \|_{\TV}  \le C_3  \Phi(\boldsymbol{x}) \bar{\delta}^k \;,
\end{equation}
for all $h<h_c$ and for all $\boldsymbol{x}$ satisfying $U( \boldsymbol{x}) < E_0$.

%... bound I_3 using Assumption (E)

To bound $I_3$, we use the characterisation of $\bar \mu$ in \eref{e:defmubar} and the definition of the total variation
distance, to get
\begin{equ}[3termbound]
\| \bar{\mu} - \mu \|_{\TV} =  2 \mu(R_h^c) \le 2 e^{-\beta E_h/2}\;,
\end{equ}
where we used Remark~\ref{rem:boundTail} to obtain the inequality.

%... combine bounds 

Combining the bounds \eqref{1termbound}, \eqref{2termbound} and \eqref{3termbound} yields
 \begin{equation} \label{kdependent}
\| P^k(\boldsymbol{x}, \cdot) - \mu \|_{\TV}  \le 
\tilde{C}_1  \Phi(\boldsymbol{x})  e^{-\beta E_h} k  + C_3  \Phi(\boldsymbol{x})   \bar{\delta}^k + 2 e^{-\beta E_h/2} \; .
\end{equation}
Since the total variation distance between a Markov chain
and its invariant measure is nonincreasing in the TV norm,
the linear dependence on $k$ can be eliminated as follows.
Set $k= \lceil h^{-1/4} \rceil$ in \eqref{kdependent} to obtain: \[
\tilde{C}_1  \Phi(\boldsymbol{x}) e^{- \beta E_h} \lceil h^{-1/4} \rceil  
+ C_3  \Phi(\boldsymbol{x})  e^{ \ln(\bar{\delta}) \lceil h^{-1/4} \rceil } + 2e^{-\beta E_h/2} \; .
\]
Since $E_h \propto h^{-1/4}$, there exist positive constants $C_1$ and $C_2$ such that
\[
\| P^k(\boldsymbol{x}, \cdot) - \mu \|_{\TV}  \le 
C_1 \Phi(\boldsymbol{x}) (  e^{-C_2/ h^{1/4}}  +    \bar  \delta^k )   \; .
\]
for all $k \in \mathbb{N}$ and every $\boldsymbol{x}$ satisfying $U( \boldsymbol{x}) < E_0$.  
This observation concludes the proof.
\end{proof}   %  of Main Result of Paper ................................................................................%

The next lemma bounds $I_1$ in \eqref{patching} using
the drift condition obtained in Lemma~\ref{patchedMALAdrift}.

\begin{lemma}
Provided that $E_\star$ is sufficiently small there exist positive
constants $C_1$, $C_2$ and $h_c$ such that 
\[
\sup_{t \in [0, T]} \| P_h^{\lfloor t/h \rfloor}(\boldsymbol{x}, \cdot)- \bar{P}_h^{\lfloor t/h \rfloor}(\boldsymbol{x}, \cdot) \|_{\TV} 
\le C_1 \Phi(\boldsymbol{x}) e^{-C_2/ h^{1/4}} (1+T)
 \]
 for all $\boldsymbol{x} \in R_h$, every $h<h_c$, and every $T>0$.
\label{1term}
\end{lemma}

\begin{proof}
The measures $P_h(\boldsymbol{x}, \cdot)$   and $\bar{P}_h(\boldsymbol{x}, \cdot)$ 
are \textit{not} the same, even for a point $\boldsymbol{x} \in R_h$, since their invariant distributions
are different.  In particular, patched MALA rejects all proposed moves to $R_h^c$.  However, if the input state and proposed 
move are in $R_h$, the acceptance probabilities of the two chains are the same.  Hence, if we initiate the 
two chains in $R_h$, and drive them by the same realization of noise, we obtain a coupling 
between the two chains such that they are identical up until the first time MALA hits $R_h^c$.  
Based on this observation, we obtain a bound on the total variation difference 
between the transition probabilities  of the two chains in the following way.

Let $\{\boldsymbol{X}_k\}$ and $\{\bar{\boldsymbol{X}}_k\}$ be instances of the Markov chains with 
respective transition probabilities $P_h$ and $\bar{P}_h$, driven by
the same realization of the noise $\boldsymbol{W}$, the same realisation of the acceptance variables $\zeta_k$,
 and with identical initial conditions
$\boldsymbol{X}_0 = \bar{\boldsymbol{X}}_0 =  \boldsymbol{x} \in R_h$. As argued above, we then
have $\boldsymbol{X}_k = \bar{\boldsymbol{X}}_k$ for $k\le n$ provided that the first time 
MALA hits $R_h^c$ is greater than $n$.  Let $\tau_h$ denote the first time that $\boldsymbol{X}_k$
hits $R_h^c$.  The coupling characterization of the total variation distance implies that, \begin{equs}
\| P_h^n(\boldsymbol{x}, \cdot)- \bar{P}_h^n(\boldsymbol{x}, \cdot) \|_{\TV} 
\le 2\P^{\boldsymbol{x}}(\boldsymbol{X}_n \ne \bar{\boldsymbol{X}}_n)  
\le 2\P^{\boldsymbol{x}}( \tau_h \le n)  \;.
\end{equs}

At this stage, one of our main ingredients is the fact that the function
$\Phi(\x) = \exp \bigl(\theta U(\x)\bigr)$ is a Lyapunov function for the MALA algorithm, see 
Proposition~\ref{MALAlocaldrift} below. The probability of MALA first hitting $R_h^c$ before time $n$  
can therefore be expressed as
\begin{equ}
 \P^{\boldsymbol{x}}  ( \tau_h \le n ) 
 = \sum_{k=1}^n \P^x\bigl( \Phi(\boldsymbol{X}_k) > e^{\theta E_h}\bigr)
 \le e^{-\theta E_h } \sum_{k=1}^n  \E^x \Phi(\boldsymbol{X}_k)
\end{equ}
where we made use of Chebychev's inequality.
We now note that we can apply
 Proposition~\ref{MALAlocaldrift} since $E_h < h^{-1/2}$ for $h$ sufficiently small. Since $E_h = E_\star h^{-1/4}$, 
 we can make $E_\star$ sufficiently small so that there exists some $\bar \gamma > 0$ such that
\begin{equ}
\E^x \Phi(\boldsymbol{X}_1)\le e^{-\bar \gamma h} \Phi(\x) + Kh \;.
\end{equ}
Combining this with the previous bound, we obtain 
\begin{equ}
 \P^{\boldsymbol{x}}  ( \tau_h \le n ) \le e^{-\theta E_h }  \sum_{k=1}^n \bigl( e^{-\bar \gamma kh}\Phi(\x) + {Kh \over 1-e^{-\bar \gamma h}}\bigr) \;.
\end{equ}
 
Summing over $k$ and using the fact that $E_h \propto h^{-1/4}$ yields the existence of positive constants $C_1$ and $C_2$ such 
that \begin{equation} \label{firsthittingtimebnd}
\P^{\boldsymbol{x}}( \tau_h \le n) 
\le    C_1 \Phi(\boldsymbol{x}) e^{- C_2/h^{1/4}} \bigl( 1 + T \bigr) \;,
\end{equation}
which is indeed the desired result.
\end{proof}

The following lemma proves a geometric rate of convergence for the Markov chain
$\bar{P}$.  Recall $R_h = \{  \boldsymbol{x} : U(\boldsymbol{x}) < E_h \}$.  
The key tool used is Harris' theorem, Theorem~\ref{theo:Harris}.

\begin{lemma}
 For every $\bar{\delta} \in (\delta, 1)$, there exist positive constants 
$C$ and $h_c$ such that
 \[
\| \bar{P}^k(\boldsymbol{x}, \cdot) - \mu \|_{\TV}  \le C \Phi(\boldsymbol{x}) \bar{\delta}^k
\]
for all $\boldsymbol{x} \in R_h$ and $h<h_c$.   
In particular, $\bar\delta$ is independent of time-stepsize.
\label{2term}
\end{lemma}

\begin{proof}
To prove this result, we use once again Harris' theorem.  
The verification of its conditions for the Markov chain $\bar{P}$ is precisely the content of 
Lemmas~\ref{patchedMALAminorization}  and~\ref{patchedMALAdrift} below. 
\end{proof}

In the next lemma, a minorization condition for patched MALA is derived
using finite time accuracy of patched MALA in the TV norm 
(see Lemma~\ref{patchedMALAtvaccuracy}).

\begin{lemma} \label{patchedMALAminorization}
Let $U$ be a potential function satisfying Assumption~\ref{sa}.
Let $\epsilon$ be the constant appearing in the minorization condition
of the true solution (see Lemma~\ref{SDEminorization}), and let
 $\bar{P}$ be as above. 
For every $E>0$ and $\bar{\epsilon} \in (0, \epsilon)$, there exists a positive
constant $h_c$ such that 
\begin{equ}[e:bounddiffP]
\| \bar{P}(\boldsymbol{x}, \cdot) - \bar{P}(\boldsymbol{y}, \cdot)\|_{\TV} \le 2(1-\bar{\epsilon}) \;,
\end{equ}
for all $\boldsymbol{x},\boldsymbol{y}$ satisfying
$U(\boldsymbol{x}) \vee U(\boldsymbol{y})\le E$ and $h \le h_c$.
\end{lemma}

\begin{proof}
According to Lemma~\ref{SDEminorization}, the bound \eref{e:bounddiffP} holds when $\bar{P}$ is 
replaced by $Q_1$, the transition probability for the true solution $\boldsymbol{Y}$ at time 
one. Combining this with Lemma~\ref{patchedMALAtvaccuracy} below, we thus obtain
\begin{equs}
\| \bar{P}(\boldsymbol{x}, \cdot) - \bar{P}(\boldsymbol{y}, \cdot)\|_{\TV} &\le 
2(1-\epsilon) 
+ 2\sup_{\Phi(\boldsymbol{x}) \le E} \| \bar{P}(\boldsymbol{x}, \cdot) - Q_1(\boldsymbol{x}, \cdot)\|_{\TV} \\
&\le 2(1-\epsilon) +  C(E) \sqrt{h}  \;.
\end{equs}
Choosing $h$ sufficiently small so 
that $C(E) \sqrt{h} < 2 ( \epsilon - \bar{\epsilon} )$, the claim follows.
\end{proof}

In the next lemma, we derive a drift condition for patched MALA using 
its single-step accuracy in representing the Lyapunov function $\Phi$.  
Deriving this drift condition requires a generalization of Theorem 7.2 
in \cite{MaStHi2002} to Lyapunov functions that are neither globally 
Lipschitz nor essentially quadratic.

\begin{lemma}\label{patchedMALAdrift}
Let $U$ be a potential function satisfying Assumption~\ref{sa} and
let $\gamma$ be the constant appearing in the drift condition \eref{e:LyapPhi}. 
For every $\bar{\gamma} \in (0, \gamma/2)$, there exist positive
constants $E_\star$ and $h_c$ such that 
  \begin{equation}
\E^{\boldsymbol{x}} \left( \Phi(\bar{\boldsymbol{X}}_{\lfloor 1/h \rfloor}) \right) \le 
e^{- \bar{\gamma}} \Phi(\boldsymbol{x})  + K\;, 
\label{MALAlocaldriftcondition}
\end{equation}
for all $\boldsymbol{x} \in R_h$ and all $h<h_c$.
\end{lemma}

\begin{proof}
We will actually show that
\begin{equ}
\E^{\boldsymbol{x}} \left( \Phi(\bar{\boldsymbol{X}}_{1}) \right) \le 
(1 - \bar{\gamma} h) \Phi(\boldsymbol{x})  + Kh\;, 
\end{equ}
from which the required bound follows by induction, noting that $U(\bar{\boldsymbol{X}}_{k}) \le E_h$
for every $k>0$ by construction.

We decompose the expression that we want to bound as
\begin{equ}
\E^{\boldsymbol{x}} \left( \Phi( \bar{\boldsymbol{X}}_1) \right) =  \E^{\boldsymbol{x}} \left( \Phi( \bar{\boldsymbol{X}}_1) \,,\; \bar{\boldsymbol{X}}_1^{\star} \in R_h \right) + 
\Phi(\boldsymbol{x}) \P^{\boldsymbol{x}} \left( \bar{\boldsymbol{X}}_1^{\star} \in R_h^c \right) \;.
\end{equ}
Since \[
\E^{\boldsymbol{x}} \left( \Phi( \bar{\boldsymbol{X}}_1) ~|~ \bar{\boldsymbol{X}}_1^{\star} \in R_h \right) =
\E^{\boldsymbol{x}} \left( \Phi( \boldsymbol{X}_1) ~|~ \boldsymbol{X}_1^{\star} \in R_h \right) \le  \E^{\boldsymbol{x}} \left( \Phi( \boldsymbol{X}_1)  \right)   \;,
\]
it follows that \begin{align*}
& \E^{\boldsymbol{x}} \left( \Phi( \bar{\boldsymbol{X}}_1) \right) \le  
 \E^{\boldsymbol{x}} \left( \Phi( \boldsymbol{X}_1)  \right)  \P^{\boldsymbol{x}} \left( \bar{\boldsymbol{X}}_1^{\star} \in R_h \right) + 
\Phi( \boldsymbol{x} )  \P^{\boldsymbol{x}} \left( \bar{\boldsymbol{X}}_1^{\star} \in R_h^c \right) \;.
\end{align*}
 Since $E_h < h^{-1/2}$ for $h$ sufficiently small, we can apply Proposition~\ref{MALAlocaldrift}
to the first term in this expression, thus obtaining
\begin{align*}
\E^{\boldsymbol{x}} \left( \Phi( \bar{\boldsymbol{X}}_1) \right) \le  
\left( 1 +  \P^{\boldsymbol{x}} \left( \bar{\boldsymbol{X}}_1^{\star} \in R_h \right)  (e^{-\gamma h} - 1 + C E_\star^4 h) \right) \Phi(\boldsymbol{x}) 
+ Kh \;.
\end{align*}
By making $E_\star$ sufficiently small, the requested bound now follows, provided that we can find a 
lower bound on $ \P^{\boldsymbol{x}} \left( \bar{\boldsymbol{X}}_1^{\star} \in R_h \right)$ that is arbitrarily close to ${1\over 2}$
for small values of $h$.

Recall that we have the identity
\begin{align*}
 \P^{\boldsymbol{x}} \left( \bar{\boldsymbol{X}}_1^{\star} \in R_h \right)  &=
 (4 \pi \beta^{-1} h)^{-n/2}  \int_{R_h}   
  \exp\left(  -  \frac{\left| \boldsymbol{y} - \boldsymbol{x} 
        + h \nabla U(\boldsymbol{x} ) \right|^2}{ 4 \beta^{-1} h} \right)  d\boldsymbol{y}  \;.
\end{align*}
Using $(a+b)^2 \le 2 a^2 + 2 b^2$ and Assumption~\ref{sa}~(D), it follows that we can bound this by
\begin{equs}
 \P^{\boldsymbol{x}} \left( \bar{\boldsymbol{X}}_1^{\star} \in R_h \right) &\ge
   (4 \pi \beta^{-1} h)^{-n/2}
  \exp\left(- \beta \frac{h E_h^2}{2} \right)  
 \int_{R_h}  \exp\left(  - \beta \frac{\left| \boldsymbol{y} - \boldsymbol{x}  \right|^2}{ 2 h} \right) 
  d\boldsymbol{y} \\
 &=  2^{-n/2} \exp\left(- \beta \frac{h^{1/2} E_\star^2}{2} \right) \P \bigl(\boldsymbol{\xi} + \x \in R_h\bigr)\;, \label{e:prefact}
 \end{equs}
where $\boldsymbol{\xi}$ denotes a Gaussian random variable with distribution $\CN(0,\beta^{-1}h)$.
In order to bound this term, denote by $\boldsymbol{n}(\x)$ the unit vector opposite the direction of the 
gradient of $U$ at $\x$, i.e.\ $\boldsymbol{n}(\x) = -\nabla U(\x) / |\nabla U(\x)|$. We claim that
for every $\delta > 0$, there exists $C>0$ and $E_0>0$ such that for every unit vector $\boldsymbol{m}$
with $\scal{\boldsymbol{m},\boldsymbol{n}(\x)} \ge \delta$, we have $U(\x + \kappa \boldsymbol{m}) \le U(\x)$,
provided that $\kappa \le C U(\x)^{-1/2}$ and $U(\x) \ge E_0$.

Indeed, consider the function $f(\kappa) =  U(\x + \kappa \boldsymbol{m}) - U(\x)$. Then $f$ is a smooth function
such that $f(0) = 0$ and $f'(0) \le - \delta |\nabla U(\x)| \le - C_1 \delta \sqrt{U(\x)}$ for some $C_1$ by Assumption~\ref{sa}.
Furthermore, one has $f''(\kappa) \le C_2 U(\x)$ for some $C_2$, as long as $f(\kappa) \le 0$.
Combining these, we see that $f'(\kappa) < 0$ (and therefore $f(\kappa) < 0$) for every $\kappa < \delta C_1 / (C_2 \sqrt{U(\x)})$,
as claimed.

For every $\x \in R_h$, we now define a set $A(\x) \subset S^{n-1}$ by $A(\x) = \{\boldsymbol{m}\,:\, \x + \kappa \boldsymbol{m} \in R_h\,\forall \kappa \le E_h^{-1}\}$.
As a consequence of our previous claim, for any $\alpha < {1\over 2}$ there exists $h_c$ such that if $h < h_c$, 
one has $\inf_{\x \in R_h} |A(\x)| / |S^{n-1}| \ge \alpha$, where $|\cdot|$ denotes the surface measure on the sphere.
Denoting by $\boldsymbol{B}(\boldsymbol{x}, r)$ the ball of radius $r$ centered 
 at $\boldsymbol{x}$, we conclude that
\begin{equs}
\P \bigl(\boldsymbol{\xi} + \x \in R_h\bigr) \ge \P \bigl(\boldsymbol{\xi} + \x \in R_h\cap \boldsymbol{B}(\boldsymbol{x}, E_h^{-1}) \bigr)
\ge \alpha \P \bigl(|\boldsymbol{\xi}| \le  h^{1/4} E_\star^{-1} \bigr)\;,
\end{equs}
where we used Assumption~\ref{sa}~(E) to obtain the last inequality.
By making $h$ sufficiently small, this expression can be made arbitrarily close to $\alpha$, and the prefactor in \eref{e:prefact}
can be made arbitrarily close to $1$, thus yielding
the required bound. 
\end{proof}

%%%%%%%%%%%%%%%%%%%%%%%%%%%%%%%%%%%%%%%%%%%%%
%%%%%%%%%%%%%%%%%%%%%%%%%%%%%%%%%%%%%%%%%%%%%

%%%%%%%%%%%%%%%%%%%%%%%%%%%%%%%%%%%%%%%%%%%%%
%%%%%%%%%%%%%%%%%%%%%%%%%%%%%%%%%%%%%%%%%%%%%

\section{Accuracy of the Patched MALA Algorithm}
\label{app:MALAtvaccuracy}

When all of the derivatives of $U$ are bounded, accuracy in the total variation distance
for forward Euler has been derived using a Talay-Tubaro expansion 
and Malliavin integration by parts \cite{BaTa1995};  
see also \cite{TaTu1990}.  In this section we treat the situation where the 
derivatives of $U$ are unbounded.   The order of accuracy
obtained below is not sharp, but the proof is constructive and is sufficient 
for MALA to inherit a minorization condition from the true solution.  
To sharpen the estimate, retrace the steps of the proof in \cite{BaTa1995} and replace 
boundedness of the coefficients by some coercivity.

\begin{lemma}\label{patchedMALAtvaccuracy}
Let $U$ be a potential satisfying Assumption~\ref{sa}.
Let $\bar{P}_h$ and $Q_h$ denote the transition probability of 
patched MALA and the true solution, respectively.
Then, for every $T>0$, there exists $C(T)>0$ 
such that for all  $h<1$, the bound
\begin{equation} \label{mainresult}
\| \bar{P}_h^{\lfloor t/h \rfloor} (\boldsymbol{x}, \cdot) - Q_{h}^{\lfloor t/h \rfloor}(\boldsymbol{x}, \cdot)   \|_{\TV} 
\le C(T)  \sqrt{h} U^3(\x)\;,
\end{equation}
is valid for all $\boldsymbol{x} \in \mathbb{R}^n$ and all $t \in [0,T]$.  
\end{lemma}

\begin{proof}
This estimate is a consequence of Lemmas~\ref{Eulertvaccuracy} 
and \ref{patchedMALAtoEuler} below.  Let  $\tilde{P}_h$ denote the transition probability of 
forward Euler \eqref{ULA}.  The triangle inequality implies that, \begin{align*}
& \| \bar{P}_h^{\lfloor t/h \rfloor}(\boldsymbol{x}, \cdot)- Q_h^{\lfloor t/h \rfloor}(\boldsymbol{x}, \cdot)   \|_{\TV}
\le \\
& \qquad \| \bar{P}_h^{\lfloor t/h \rfloor}(\boldsymbol{x}, \cdot)- \tilde{P}_h^{\lfloor t/h \rfloor}(\boldsymbol{x}, \cdot)   \|_{\TV} + 
\| \tilde{P}_h^{\lfloor t/h \rfloor}(\boldsymbol{x}, \cdot)- Q_h^{\lfloor t/h \rfloor}(\boldsymbol{x}, \cdot)   \|_{\TV} \;.
\end{align*}
According to Lemma~\ref{patchedMALAtoEuler}, the first term is bounded by $C(T)  \sqrt{h} U^3(\x)$.  
According to Lemma~\ref{Eulertvaccuracy}, the second term is bounded by 
$C(T) \sqrt{h} U^2(\x)$.   Hence, the desired error estimate is obtained.
\end{proof}

\begin{lemma}  \label{Eulertvaccuracy}
Let $U$ be a potential satisfying Assumption~\ref{sa}.
Let $\tilde{P}_h$ and $Q_h$ denote the transition probability of 
forward Euler and the true solution, respectively.
Then, for every $T>0$, there exists $C(T)>0$ 
such that for all  $h<1$ 
\[
\| \tilde{P}_h^{\lfloor t/h \rfloor}(\boldsymbol{x}, \cdot)- Q_h^{\lfloor t/h \rfloor}(\boldsymbol{x}, \cdot)   \|_{\TV}
\le C(T) \sqrt{h} U^2(\x)\;,
\]
for all $\boldsymbol{x} \in \mathbb{R}^n$ and all $t \in [0,T]$.  
\end{lemma}

\begin{proof}
We bound the TV distance between forward Euler and the true solution 
using Lemmas~\ref{EulerStrongAccuracy},~\ref{EulerLocalLipschitz}, and~\ref{EulerWeak} as follows.
Using the triangle inequality, we split the quantity that we wish to bound as
\begin{equs}
\| \tilde{P}_h^{\lfloor t/h \rfloor}(\boldsymbol{x}, \cdot)- Q_h^{\lfloor t/h \rfloor}(\boldsymbol{x}, \cdot)   \|_{\TV} &\le 
 \| \tilde{P}_h^{\lfloor t/h \rfloor}(\boldsymbol{x}, \cdot)- (\tilde{P}_h \circ Q_h^{\lfloor t/h \rfloor - 1})(\boldsymbol{x}, \cdot)   \|_{\TV} \\
 &\qquad +
\| (\tilde{P}_h \circ Q_h^{\lfloor t/h \rfloor - 1})(\boldsymbol{x}, \cdot) - Q_h^{\lfloor t/h \rfloor}(\boldsymbol{x}, \cdot)   \|_{\TV}\\
&\eqdef I_1 + I_2\;.
\label{Eulertvbound}
\end{equs}
We can rewrite the first term of \eqref{Eulertvbound} as
\[
I_1 =
\E^{\boldsymbol{x}} \| \tilde{P}_h(\tilde{\boldsymbol{X}}_{\lfloor t/h \rfloor - 1}, \cdot) - 
\tilde{P}_h(\boldsymbol{Y}_{\lfloor t/h \rfloor -1}, \cdot) \|_{\TV} \;,
\]
which, using Lemma~\ref{EulerLocalLipschitz},  is bounded by 
\begin{align*}
& I_1 \le \frac{1}{\sqrt{2 \beta^{-1} h}}
 \E^{\boldsymbol{x}} 
 \big( | \tilde{\boldsymbol{X}}_{\lfloor t/h \rfloor - 1} - \boldsymbol{Y}_{\lfloor t/h \rfloor -1} | \wedge 1 \big)  \\
&\qquad + \sqrt{\frac{h}{2 \beta^{-1}}}
 \E^{\boldsymbol{x}} 
 \big( | \nabla U(\tilde{\boldsymbol{X}}_{\lfloor t/h \rfloor - 1}) - \nabla U(\boldsymbol{Y}_{\lfloor t/h \rfloor -1}) | \wedge 1 \big)  \;.
\end{align*}
Strong accuracy of forward Euler in a bounded metric (see Lemma~\ref{EulerStrongAccuracy}) then yields 
\begin{align*}
I_1 \le C \sqrt{h} U^2(\x)\;.
\end{align*}

The second term  of \eqref{Eulertvbound} is bounded by 
\begin{equ}
I_2 \le
\E^{\boldsymbol{x}} \big( \| \tilde{P}_h(\boldsymbol{Y}_{\lfloor t/h \rfloor - 1}, \cdot) - 
Q_h(\boldsymbol{Y}_{\lfloor t/h \rfloor -1}, \cdot) \|_{\TV} \big) \;.
\end{equ}
From Lemma~\ref{EulerWeak} and \eref{e:LyapPhi}, it follows that $I_2$ is bounded by $C h U^2(\x)$,
and the claim follows.
\end{proof}

Even though forward Euler is numerically unstable for drifts that are not globally
Lipschitz,  one can prove the following `strong accuracy' for forward Euler 
in a bounded metric.  As the proof shows, boundedness of the metric plays the 
role of stability of the numerical scheme.

\begin{lemma} \label{EulerStrongAccuracy}
Let $U$ be a potential satisfying Assumption~\ref{sa}.
Let $\tilde{\boldsymbol{X}}$ and $\boldsymbol{Y}$ denote 
forward Euler and the true solution, respectively.
Then, for every $T>0$ there exists $C(T)>0$  such that
\[
 \E^{\boldsymbol{x}} \bigl( 
 | \tilde{\boldsymbol{X}}_{\lfloor t/h \rfloor} -  \boldsymbol{Y}(\lfloor t/h \rfloor h) | \wedge 1 \bigr)
\le C(T) h U^2(\x)    \;,
\]
holds for all $\boldsymbol{x} \in \mathbb{R}^n$, all $h \le 1$, and all $t \in [0,T]$.  
\end{lemma}

\begin{proof}
The proof goes by induction over the number of steps, so let us consider one single step first.
We then have
\begin{equ}
\tilde{\boldsymbol{X}}_1 -  \Y_h = \tilde{\boldsymbol{X}}_0 -  \Y_0 - \int_0^h \bigl(\nabla U(\tilde{\boldsymbol{X}}_0) - \nabla U(\Y_s)\bigr)\,ds\;,
\end{equ}
so that 
\begin{equs}
|\tX_1 -  \Y_h|^2 &=  |\tX_0 - \Y_0|^2 - 2h \scal{\tX_0 - \Y_0, \nabla U(\tX_0) - \nabla U(\Y_0)} \\
&\quad + 2 \int_0^h \scal{\tX_0 - \Y_0, \nabla U(\Y_s) - \nabla U(\Y_0)}\,ds \\
&\quad + \Bigl|\int_0^h \bigl(\nabla U(\tX_0) - \nabla U(\Y_s)\bigr)\,ds\Bigr|^2\;.
\end{equs}
Together with Remark~\ref{oslremark}, this implies that there exists a constant $C$ such that 
\begin{equs}
|\tX_1 -  \Y_h|^2 &\le (1+ Ch) |\tX_0 - \Y_0|^2 + 2h^2 |\nabla U(\tilde{\boldsymbol{X}}_0) - \nabla U(\Y_0)|^2 \\
&\quad + 2 \int_0^h \scal{\tX_0 - \Y_0, \nabla U(\Y_s) - \nabla U(\Y_0)}\,ds \\
&\quad + 2h\int_0^h \bigl|\nabla U(\Y_s) - \nabla U(\Y_0)\bigr|^2\,ds\;. \label{e:bounddiffEuler}
\end{equs}
Note now that if $\eta$ is any unit vector in $\R^n$, we have the identity
\begin{equs}
\scal{\nabla U(\Y_s) - \nabla U(\Y_0), \eta} &= \int_0^s \CL \scal{\nabla U, \eta}(\Y_r)\,dr 
 + \sqrt{2\over \beta} \int_0^s D^2 U(\Y_r)(\eta, d\W_r)\;.
\end{equs}
Since $\|D^2 U\| \le C U$ and $|\CL \scal{\nabla U, \eta}| \le C U^2$, it then follows from Remark~\ref{SDEdriftcondition2}
that there exists a constant $C$ such that 
\begin{equs}
\E |\nabla U(\Y_s) - \nabla U(\Y_0)|^2 &\le C s U^4(\Y_0)\;,\quad \forall s \le 1\;,\\
|\E \scal{\eta,\nabla U(\Y_s) - \nabla U(\Y_0)}| &\le C s U^2(\Y_0)\;,\quad \forall s \le 1\;.
\end{equs}
On the other hand, one also has the bound
\begin{equ}
\bigl|\nabla U(\tilde{\boldsymbol{X}}_0) - \nabla U(\Y_0)\bigr|^2 \le C |\tX_0 - \Y_0|^2 \exp \bigl(C |\tX_0 - \Y_0|\bigr) U^2(\Y_0)\;,
\end{equ}
which follows from Assumption~\ref{sa}~(D) and Lemma~\ref{lem:Uexp} below. In the case where $|\tX_0 - \Y_0| \le 1$, this yields
\begin{equ}
\bigl|\nabla U(\tilde{\boldsymbol{X}}_0) - \nabla U(\Y_0)\bigr|^2 \le h^{-1} |\tX_0 - \Y_0|^2 + h C U^4(\Y_0)\;.
\end{equ}
Inserting these bounds into
\eref{e:bounddiffEuler}, we see that there is $C>0$ such that if $|\tX_0 - \Y_0| \le 1$, then
\begin{equ}
\E |\tX_1 -  \Y_h|^2 \le (1+Ch)|\tX_0 - \Y_0|^2 + Ch^3 U^4(\Y_0)\;.
\end{equ}
Since on the other hand, one obviously has $\E \bigl(|\tX_1 -  \Y_h|^2\wedge 1\bigr) \le 1$, we conclude that 
\begin{equ}
\E \bigl(|\tX_1 -  \Y_h|^2\wedge 1\bigr) \le (1+ Ch)\bigl(|\tX_0 - \Y_0|^2\wedge 1\bigr) + Ch^3 U^4(\Y_0)\;.
\end{equ}
The requested bound now follows from the \textit{a priori} bounds on the solution $\Y_t$ given by
Remark~\ref{SDEdriftcondition2}.
\end{proof}

\begin{lemma}  \label{EulerLocalLipschitz}
Let $U$ be a potential satisfying Assumption~\ref{sa}.
Let $\tilde{P}_h$ denote the transition probability of forward Euler.
For every $h<1$ and for all $\boldsymbol{x}, \boldsymbol{y} \in \mathbb{R}^n$,
 \[
\| \tilde{P}_h(\boldsymbol{x}, \cdot) - \tilde{P}_h(\boldsymbol{y}, \cdot) \|_{\TV}
\le   \frac{1}{\sqrt{2 \beta^{-1} h}}  | \boldsymbol{x} - \boldsymbol{y} | + 
   \frac{\sqrt{h}}{\sqrt{2 \beta^{-1}}}  | \nabla U(\boldsymbol{x} ) - \nabla U(\boldsymbol{y} ) |  \;.
\]
\end{lemma}

\begin{proof}
Recalling Pinsker's inequality:
\begin{equ}
\|\mathcal{N}(0,\sigma) - \mathcal{N}(\x, \sigma)\|_\TV \le {| \x|\over \sqrt \sigma} \;,
\end{equ}
we see that the claim follows from the fact that 
\begin{equ}
\tilde{P}_h(\boldsymbol{x}, \cdot) = \mathcal{N}( \boldsymbol{x} - h \nabla U(\boldsymbol{x} ), 2 \beta^{-1} h \mathbf{I})\;,
\end{equ}
where $\mathbf{I}$ denotes the identity matrix.
\end{proof}

\begin{lemma}  \label{EulerWeak}
Let $U$ be a potential satisfying Assumption~\ref{sa}.
Let $\tilde{P}_h$ and $Q_h$ denote the transition probability of forward Euler
and the true solution, respectively.
Then, there exists $C>0$ such that, for every  $h<1$, the bound
 \[
\| \tilde{P}_h(\boldsymbol{x}, \cdot) - Q_h(\boldsymbol{x}, \cdot) \|_{\TV}
\le C h U^2(\x) \;,
\]
holds for all $\boldsymbol{x} \in \mathbb{R}^n$.  
\end{lemma}

\begin{proof}
We write $E_0 = U(\x)$ as a shorthand. The bound is trivial if $E_0^2 h \ge 1$, so we can and will assume in the sequel that  $E_0^2 h \le 1$.
Recall that the transition probabilities $Q_h$ are generated by the solutions at time $h$ to
\begin{equ}[e:truesol]
d \boldsymbol{Y} = -\nabla U \bigl( \boldsymbol{Y} \bigr)\,dt 
+ \sqrt{2 \beta^{-1}} d \boldsymbol{W}\;,\qquad \boldsymbol{Y}(0) = \boldsymbol{x}\;,
\end{equ}
whereas the transition probabilities $\tilde{P}_h$ of forward Euler can be interpreted as the solution
at time $h$ to
\begin{equ}[e:Eulersol]
d \tilde{\boldsymbol{X}} = -\nabla U \bigl( \boldsymbol{x} \bigr)\,dt 
+ \sqrt{2 \beta^{-1}} d \boldsymbol{W}\;,\qquad \tilde{\boldsymbol{X}}(0) = \boldsymbol{x}\;.
\end{equ}
Therefore, the required quantity can be bounded from above by the total variation distance between
the measures generated by \eref{e:truesol} and \eref{e:Eulersol} on pathspace between times $0$ and $h$.
Since only the drift differs in the SDEs \eref{e:truesol} and \eref{e:Eulersol}, Girsanov's theorem can be used
to quantify the distance between the laws of the solutions at time $h$ 
to \eref{e:truesol} and \eref{e:Eulersol}.

We first replace the potential $U$ by a modified potential $\tilde U$ which is bounded, together with all
of its derivatives. Indeed, let $\phi\colon \R_+ \to \R$ be a smooth increasing function such that $\phi(x) = x$
for $x \le 2$ and $\phi(x) = 3$ for $x \ge 4$. With this definition at hand, we set
\begin{equ}
\tilde U( \boldsymbol{y}) = U(\boldsymbol{x})\phi\bigl(U( \boldsymbol{y})/U(\boldsymbol{x})\bigr)\;.
\end{equ}
It then follows from Assumption~\ref{sa}~(D) that there exists a constant $C$ such that
\begin{equ}[e:boundUtilde]
|\tilde U(\boldsymbol{y})| + \|D\tilde U(\boldsymbol{y})\| + \|D^2\tilde U(\boldsymbol{y})\| \le C E_0 \;,
\end{equ}
uniformly over all $\boldsymbol{y} \in \R^n$. 
%Before we proceed, we remark that it follows from Assumption~\ref{sa}~(D) that for every
%$K>0$ there exists a $c>0$ such that the bound
%\begin{equ}[e:diffU]
%|U(x) - U(y)| \le K U(x)\;,
%\end{equ}
%holds for all $x$ and for all $y$ such that $|x-y| \le c$.

Before we proceed, we argue that if we define 
\begin{equ}[e:truetilde]
d \tilde{\boldsymbol{Y}} = -\nabla \tilde{U} \bigl( \tilde {\boldsymbol{Y}} \bigr)\,dt 
+ \sqrt{2 \beta^{-1}} d \boldsymbol{W}\;,\qquad \tilde{\boldsymbol{Y}}(0) = \boldsymbol{x}\;,
\end{equ}
then, one has 
$\P \bigl(\exists t \le h \,:\, \boldsymbol{Y}(t) \neq \tilde{\boldsymbol{Y}}(t)\bigr) \le C E_0^2 h$,
so that we can replace $U$ by $\tilde U$ in \eref{e:truesol} without any loss of generality. 
In order to show this, we note that Lemma~\ref{lem:driftGen} yields the existence of a constant $K$ such that
$M(t) = U(\boldsymbol{Y}(t)) - Kt - U(\x)$ is a supermartingale with quadratic variation process
\begin{equ}
\scal{M,M}(t) = 2\beta^{-1}\int_0^t |\nabla U(\boldsymbol{Y}(s))|^2\,ds\;.
\end{equ}
Furthermore, for $E_0$ sufficiently large (independently of $h$), one has 
\begin{equ}
\P \bigl(\exists t \le h \,:\, \boldsymbol{Y}(t) \neq \tilde{\boldsymbol{Y}}(t)\bigr) 
\le \P\bigl(\sup_{t \le h} M_t \ge {\textstyle{1\over2}} U(\x)\bigr)\;.
\end{equ}
It then follows from the
exponential martingale inequality \cite[p.~153]{Yor} that, for every $\Lambda > 0$, one has the bound 
\begin{equ}
\P \bigl(\exists t \le h \,:\, \boldsymbol{Y}(t) \neq \tilde{\boldsymbol{Y}}(t)\bigr) 
\le \exp\bigl(-U^2(\x)/(8\Lambda)\bigr) + \P \bigl(\scal{M,M}(h) \ge \Lambda\bigr)\;.
\end{equ}
For $\delta > 0$ sufficiently small, the second term in this 
expression can then be bounded by
\begin{equs}
\P \bigl(\scal{M,M}(h) \ge \Lambda\bigr) &\le \exp \bigl(-\sqrt{\delta \Lambda h^{-1}}\bigr) \E \exp \sqrt{\delta h^{-1} \scal{M,M}(h)}\\
 &\le \exp \bigl(-\sqrt{\delta \Lambda h^{-1}}\bigr) {1\over h} \int_0^h \E \exp \sqrt{2\beta^{-1}\delta |\nabla U(\boldsymbol{Y}(s))|^2}\,ds\\
 &\le \exp \bigl(-\sqrt{\delta \Lambda h^{-1}}\bigr) {1\over h} \int_0^h \E \exp \bigl(C \sqrt \delta U(\boldsymbol{Y}(s))\bigr)\,ds\\
 &\le C\exp \bigl(-\sqrt{\delta \Lambda h^{-1}}\bigr) \exp \bigl(C \sqrt \delta U(\x)\bigr)\;.
\end{equs}
Here, we have first used Chebychev's inequality, followed by Jensen's inequality, then 
Assumption~\ref{sa}~(D), and finally Lemma~\ref{lem:driftGen}
with $\delta$ small enough. Setting $\Lambda = U^2(\x) h^{-1/3}$, it follows that for $h$ small enough
we actually have $\P \bigl(\exists t \le h \,:\, \boldsymbol{Y}(t) \neq \tilde{\boldsymbol{Y}}(t)\bigr) \le 2 \exp(-c h^{-1/3})$
for some positive constant $c$, which is much better than needed.

We now proceed by comparing the true solution and forward Euler for $\tilde{U}$.
Denote now by $\CQ_h$ the measure on pathspace generated by \eref{e:truetilde}, by $\CP_h$ the measure
on pathspace generated by solutions to \eref{e:Eulersol}, and by $\CW_h$ Wiener measure on $\CC([0,h], \R^d)$ 
with starting point $\boldsymbol{x}$.  It then follows from Girsanov's theorem that
\begin{equs}
{d\CQ_h \over d\CW_h}(\boldsymbol{W}) &= 
Z_Q^{-1} \exp \Bigl(
- \frac{1}{\sqrt{2 \beta^{-1}}}  ( \tilde U(\boldsymbol{W}_h) - \tilde U(\boldsymbol{x})  ) 
- \beta \int_0^h G(\boldsymbol{W}_t)\,dt\Bigr)\;, \\
{d\CP_h \over d\CW_h}(\boldsymbol{W}) &= 
Z_P^{-1} \exp \Bigl(
- \frac{1}{\sqrt{2 \beta^{-1}}}  \nabla \tilde U(\boldsymbol{x})^T (\boldsymbol{W}_h - \boldsymbol{x}) 
- h \beta |\nabla \tilde U(\boldsymbol{x})|^2\Bigr)\;,
\end{equs}
for some normalisation factors $Z_P$ and $Z_Q$, where the function $G$ is given by
\begin{equ}
G(x) =  |\nabla \tilde U(x )|^2 -  \Delta \tilde U(x)  \;.
\end{equ}
(See for example \cite[Theorem~11A]{David}.) In particular, we have
\begin{equs}
{d\CQ_h \over d\CP_h}(\boldsymbol{W}) &= Z_h^{-1}  \exp \Bigl(- 
\frac{1}{\sqrt{2 \beta^{-1}}}  \bigl(
\tilde U(\boldsymbol{W}_h) - \tilde U(\boldsymbol{x}) 
-  \nabla \tilde U(\boldsymbol{x})^T (\boldsymbol{W}_h - \boldsymbol{x})\bigr) \\
&\qquad - \int_0^h \beta \Bigl(G(\boldsymbol{W}_t) - |\nabla \tilde U(\boldsymbol{x})|^2\Bigr)\,dt\Bigr) \eqdef 
Z_h^{-1}\exp\bigl(\CD_h(\boldsymbol{W})\bigr)\;,
\end{equs}
where the normalisation constant $Z_h$ is given by
\begin{equ}
Z_h = \int \exp(\CD_h)\,d\CP_h\;.
\end{equ}
By \eref{e:boundUtilde}, there exists a constant $C>0$ such that the 
bound \begin{equ}[e:boundDW]
|\CD_h(\boldsymbol{W})| \le C E_0\bigl(|\boldsymbol{W}_h-\boldsymbol{x}|^2  + E_0 h\bigr)\;,
\end{equ}
holds for every $\boldsymbol{W}$. As an immediate consequence, for every $c>0$, there exists a constant $C>0$ such that 
\begin{equ}
\Bigl|\log \int \exp \bigl(c\CD_h\bigr)\,d\CP_h\Bigr| \le C E_0^2 h
\end{equ}
for every $h \le 1$. In particular, one has $Z_h = 1+ \CO(E_0^2 h)$ and similarly for $Z_h^{-1}$.
Denote now by $B_h$ the set
\begin{equ}
B_h = \{\boldsymbol{W}\,:\, |\CD_h(\boldsymbol{W})| \ge 1\}\;.
\end{equ}
It follows from the bound \eref{e:boundDW} that $\CP_h(B_h) \le C\exp\bigl(-c/(hE_0)\bigr)$ for some $c,C>0$ and for $h E_0^2 \le 1$.

We conclude that
\begin{equs}
\|\CQ_h - \CP_h\|_{\TV} &= \int |1 - Z_h^{-1} \exp(\CD_h)|\, d\CP_h \le 
C\int_{B_h^c} |\CD_h|\, d\CP_h + \CO(E_0^2h) \\
&\qquad + \int_{B_h} |1 - Z_h^{-1} \exp(\CD_h)|\, d\CP_h \\
&\le \CO(E_0^2h) + \sqrt{{\CP_h(B_h) \over Z_h^{2}}\Bigl(\int \exp(2\CD_h)\, d\CP_h - 1\Bigr)}
= \CO(E_0^2h)\;,
\end{equs}
as required.  In the last step, we have used the Cauchy-Schwarz inequality.
\end{proof}

\begin{lemma}
For every $T>0$,
there exists a $C(T)>0$ such that 
\[
\sup_{t \in [0, T]} 
\| \bar{P}_h^{\lfloor t/h \rfloor}(\boldsymbol{x}, \cdot) - \tilde{P}_h^{\lfloor t/h \rfloor}(\boldsymbol{x}, \cdot) \|_{\TV}
\le C(T) \sqrt{h}    U^3(\x)
\]
holds  for every $h<1$ and for every $\boldsymbol{x} \in \mathbb{R}^n$.  
\label{patchedMALAtoEuler}
\end{lemma}

\begin{proof}
Denote by $\tilde{\boldsymbol{X}}_{k}$ the solution to the forward Euler algorithm after $k$ steps and by
$\boldsymbol{X}_{k}$ the solution to the MALA algorithm. Since both agree until the first time
that one step is rejected, it follows from the coupling inequality that we have the bound
\begin{equ}
\| \bar{P}_h^{n}(\boldsymbol{x}, \cdot) - \tilde{P}_h^{n}(\boldsymbol{x}, \cdot) \|_{\TV}
\le 2 \sum_{k=0}^{n-1} \E^\x \bigl|1-\alpha_h(\tilde{\boldsymbol{X}}_{k}, \tilde{\boldsymbol{X}}_{k+1})\bigr|\;.
\end{equ}
At this stage, we note that since $\alpha_h \in [0,1]$, it follows from Lemma~\ref{MALAstagnationprobability} 
that for every $\alpha > 0$ there exists a $C>0$ such that the bound
\begin{equ}
\E^\x \bigl|1-\alpha_h(\x, \tilde{\boldsymbol{X}}_{1})\bigr| 
\le  C h^{3/2} \bigl(U(\x) \wedge \alpha h^{-1/2})^3\;,
\end{equ}
holds for \textit{all} $\x \in \R^n$. This is simply because this bound is trivial for $U(\x) > 1/\sqrt h$.

Making $\alpha$ sufficiently small and combining this with Corollary~\ref{globalEuler}, we then obtain
\begin{equs}
\E^\x \bigl|1-\alpha_h(\tilde{\boldsymbol{X}}_{k}, \tilde{\boldsymbol{X}}_{k+1})\bigr| 
&\le  C h^{3/2} \E^\x \bigl(U(\tilde{\boldsymbol{X}}_{k}) \wedge \alpha h^{-1/2})^3 \\
&\le  C h^{3/2} \bigl(U^3(\x) + Khk) \;,
\end{equs}
for some constant $K>0$. The claim now follows at once by summing over $k$.
\end{proof}

%%%%%%%%%%%%%%%%%%%%%%%%%%%%%%%%%%%%%%%%%%%%%
%%%%%%%%%%%%%%%%%%%%%%%%%%%%%%%%%%%%%%%%%%%%%

\section{Local Drift Conditions}
\label{app:localdriftcondition}

This section shows that the single-step accuracy of MALA and forward Euler imply that 
these algorithms preserve Lyapunov functions of the true solution locally.   
We refer to this property of a numerical method as a local drift condition.  In the lemmas that 
follow local drift  conditions are derived for the MALA and forward Euler algorithms.  
Deriving such drift conditions requires adapting Theorem 7.2 
of \cite{MaStHi2002} to Lyapunov functions that are neither globally 
Lipschitz nor essentially quadratic.   Still, the proofs in this section are strongly inspired
by the results in \cite{MaStHi2002}.

A key technical issue addressed below is that the natural Lyapunov function of the true solution,
namely $\Phi(\x) = \exp(\theta U(\x))$ grows so fast that it is not in general integrable with respect to
a Gaussian measure. In particular, it is not integrable with respect to  the transition probabilities of forward Euler.   
Nevertheless, 
we will show that the expectation of $\Phi$ under one step of MALA is finite and close to 
the expectation of $\Phi$ under the true solution.  Integrability of $\Phi$ with respect to the 
transition probability of MALA is a consequence of MALA preserving an equilibrium 
measure whose tails are lighter than Gaussian.

A first remark which will be useful in this section is that under our assumptions on the potential $U$,
it does not behave `worse than exponential' in the following sense:

\begin{lemma}\label{lem:Uexp}
There exists $C>0$ such that for every $\boldsymbol{x}, \boldsymbol{y} \in \R^n$, we have
\begin{equ}
|U( \boldsymbol{x} )| \le |U( \boldsymbol{y})| \exp \bigl( C | \boldsymbol{x}- \boldsymbol{y} |\bigr)\;.
\end{equ}
\end{lemma}

\begin{proof}
It suffices to differentiate the function $t \mapsto U\bigl((1-t) \boldsymbol{x} + t \boldsymbol{y}\bigr)$,
invoke Assumption~\ref{sa}~(D), and apply Gronwall's inequality over the interval $t \in [0,1]$.
\end{proof}

%%%%%%%%%%%%%%%%%%%%%%%%%%%%%%%%%%%%%%%%%%%%%

\begin{proposition}\label{MALAlocaldrift}
Set $\Phi(\boldsymbol{x}) = \exp(\theta U(\boldsymbol{x}) )$.  Let
$\boldsymbol{X}_1$ denote MALA after one step.  Then there exist positive constants $C$ 
and $\theta \in (0, \beta)$ such that the bound
\[
 \E^{\boldsymbol{x}} \left( \Phi(\boldsymbol{X}_1) \right)  \le 
 ( e^{-\gamma h} + C U^4(\boldsymbol{x}) h^2) \Phi(\boldsymbol{x}) 
 + \frac{K}{\gamma} (1- e^{-\gamma h})
 \]
holds for all $\boldsymbol{x} \in \R^n$ satisfying $U(\boldsymbol{x}) < h^{-1/2}$.
\end{proposition}

\begin{proof}
Denoting by $\boldsymbol{Y}(h)$ the true solution after time $h$, we write   \[
 \E^{\boldsymbol{x}} \left( \Phi(\boldsymbol{X}_1) \right) =
  \E^{\boldsymbol{x}} \left( \Phi(\boldsymbol{Y}(h)) \right) +
  \E^{\boldsymbol{x}} \left( \Phi(\boldsymbol{X}_1) - \Phi(\boldsymbol{Y}(h)) \right) \;.
\]
We know from \eref{e:LyapPhi} that $\Phi$ is a Lyapunov function
for the true solution, and hence, \[
 \E^{\boldsymbol{x}} \left( \Phi(\boldsymbol{X}_1) \right) \le
  e^{-\gamma h} \Phi(\boldsymbol{x}) + \frac{K}{\gamma} (1-e^{-\gamma h}) +
 | \E^{\boldsymbol{x}} \left( \Phi(\boldsymbol{X}_1) - \Phi(\boldsymbol{Y}(h)) \right) | \;.
\]
The approximation result between MALA and the true solution
given in Lemma~\ref{MALAphiaccuracy} below then implies the desired result.
\end{proof}

The following lemma states that the single step error of MALA in 
preserving $\Phi$ is $\CO(h^2)$ with an error constant that depends 
on $\Phi(\boldsymbol{y})$ and $U^4(\boldsymbol{y})$ evaluated at the 
initial condition.

%%%%%%%%%%%%%%%%%%%%%%%%%%%%%%%%%%%%%%%%%%%%%

\begin{lemma} \label{MALAphiaccuracy}
Set $\Phi(\boldsymbol{x}) = \exp(\theta U(\boldsymbol{x}) )$.  Let
$\boldsymbol{X}_1$ and $\boldsymbol{Y}(h)$ denote MALA and the true
solution after one step, respectively.  Then there exist positive constants $C$ 
and $\theta \in (0, \beta)$ such that the bound 
\begin{equ}[e:errorPhiMALA]
| \E^{\boldsymbol{x}} \left( \Phi(\boldsymbol{X}_1) - \Phi(\boldsymbol{Y}(h)) \right) |
 \le C U^4(\boldsymbol{x})  \Phi(\boldsymbol{x})   h^2
\end{equ}
 holds for all $\boldsymbol{x} \in \R^n$  satisfying $U(\boldsymbol{x}) < h^{-1/2}$.
\end{lemma}

\begin{remark}
Note in particular that the bound \eref{e:errorPhiMALA} implies that
$\E^{\boldsymbol{x}}  \Phi(\boldsymbol{X}_1) < \infty$. This is not obvious \textit{a priori} since 
$\Phi(\x)$ grows faster than $\exp |\x|^2$ at infinity. As a consequence, this expectation is infinite 
under the proposal moves.
\end{remark}

\begin{proof}
Applying It\^o's formula twice to the exact solution yields
\begin{equ} \label{EXACTphiexpansion}
\E^{\boldsymbol{x}}  \left( \Phi(\boldsymbol{Y}(h))  \right)  = 
\Phi(\boldsymbol{x})
+ h (\mathcal{L} \Phi) (\boldsymbol{x}) 
+ h^2 \int_0^1 (1-t)\,\E^{\boldsymbol{x}} \left( \mathcal{L}^2 \Phi (\boldsymbol{Y}(h t)) \right) 
dt  \;,
\end{equ}
where $\CL$ denotes the generator as in \eref{generator}.

Setting $\boldsymbol{X}(s) = \boldsymbol{x} + s (\boldsymbol{X}_1^{\star} - \boldsymbol{x})$, we 
obtain by a simple application of Taylor's formula the following identity for the
application of one step of MALA: 
\begin{align} \label{MALAphiexpansion}
& \E^{\boldsymbol{x}}  \left( \Phi(\boldsymbol{X}_1)  \right)   = \Phi(\boldsymbol{x}) \\
& \qquad + 
\E^{\boldsymbol{x}} \left( D \Phi(\boldsymbol{x}) (\boldsymbol{X}_1^{\star} - \boldsymbol{x}) \alpha_h(\boldsymbol{x}, \boldsymbol{X}_1^{\star}) \right) 
\nonumber \\
& \qquad + \frac{1}{2} 
\E^{\boldsymbol{x}} \left( D^2 \Phi(\boldsymbol{x}) (\boldsymbol{X}_1^{\star} - \boldsymbol{x}, \boldsymbol{X}_1^{\star} - \boldsymbol{x}) 
\alpha_h(\boldsymbol{x}, \boldsymbol{X}_1^{\star}) \right) 
\nonumber  \\
& \qquad + 
\frac{1}{6} \E^{\boldsymbol{x}} \left( D^3 \Phi(\boldsymbol{x}) (\boldsymbol{X}_1^{\star} - \boldsymbol{x}, \boldsymbol{X}_1^{\star} - \boldsymbol{x},\boldsymbol{X}_1^{\star} - \boldsymbol{x} ) \alpha_h(\boldsymbol{x}, \boldsymbol{X}_1^{\star}) \right)  
\nonumber \\
& \qquad + 
{1\over 4!} \int_0^1(1-t)^4 \E^{\boldsymbol{x}} \left( D^4 \Phi(\boldsymbol{X}(t)) (\boldsymbol{X}_1^{\star} - \boldsymbol{x} )^4 \alpha_h(\boldsymbol{x}, \boldsymbol{X}_1^{\star}) \right) \, dt\;.
\nonumber
\end{align}
(Here we interpret $D^4 \Phi ( \y) ( \x )^4$
as being the quadrilinear form $D^4 \Phi ( \y)$ applied to  $( \x, \x, \x, \x)$.)  
Subtracting  \eqref{EXACTphiexpansion} from \eqref{MALAphiexpansion} and using the definition of the forward 
Euler proposal move to collect this difference in powers of $h$, we obtain:
 \begin{equation} \label{MALAphierror}
 \E^{\boldsymbol{x}}  \left( \Phi(\boldsymbol{X}_1)  - \Phi(\boldsymbol{Y}(h)) \right)  = 
 h^{1/2} I_{1/2}  + h I_1 + h^{3/2} I_{3/2} + h^2 I_2  + R_2 \;.
 \end{equation}
 Here we have introduced:
 \begin{equs}
I_{1/2} =&  
\sqrt{2 \beta^{-1}} \E^{\boldsymbol{x}} \left( \alpha_h(\boldsymbol{x}, \boldsymbol{X}_1^{\star}) 
 D \Phi(\boldsymbol{x}) \boldsymbol{\xi}  \right) \\
I_1 =& -   \theta \Phi(\boldsymbol{x}) | \nabla U(\boldsymbol{x}) |^2  \E^{\boldsymbol{x}} \left( \alpha_h(\boldsymbol{x}, \boldsymbol{X}_1^{\star}) - 1 \right)  \\
& +  \beta^{-1} \E^{\boldsymbol{x}} \left( D^2 \Phi(\boldsymbol{x}) (\boldsymbol{\xi}, \boldsymbol{\xi}) (\alpha_h(\boldsymbol{x}, \boldsymbol{X}_1^{\star}) -1) \right)  
 \\
I_{3/2}= & 
 -  \sqrt{2 \beta^{-1}} \E^{\boldsymbol{x}} \left( D^2 \Phi(\boldsymbol{x}) (\boldsymbol{\xi}, \nabla U(\boldsymbol{x}) ) \alpha_h(\boldsymbol{x}, \boldsymbol{X}_1^{\star})  \right) \\  
&+ \frac{1}{6} \left( (2 \beta^{-1})^{3/2} \E^{\boldsymbol{x}} \left( D^3 \Phi(\boldsymbol{x}) (\boldsymbol{\xi}, \boldsymbol{\xi}, \boldsymbol{\xi})
    \alpha_h(\boldsymbol{x}, \boldsymbol{X}_1^{\star})  \right)  \right)
\\ 
I_2 = & \frac{1}{2}  \left(
\E^{\boldsymbol{x}} \left( D^2 \Phi(\boldsymbol{x}) (\nabla U(\boldsymbol{x}), \nabla U(\boldsymbol{x}) )
\alpha_h(\boldsymbol{x}, \boldsymbol{X}_1^{\star}) \right)  \right) 
 \\
& - \frac{1}{6} \left( \E^{\boldsymbol{x}} \left( D^3 \Phi(\boldsymbol{x}) ( \nabla U(\boldsymbol{x}), 
(-\sqrt{h} \nabla U(\boldsymbol{x}) + \sqrt{2 \beta^{-1}} \boldsymbol{\xi})^2) \alpha_h(\boldsymbol{x}, \boldsymbol{X}_1^{\star}) \right)  \right)
 \\
&  - 
\frac{1}{6} \left( \sqrt{2 \beta^{-1}} \E^{\boldsymbol{x}} \left( D^3 \Phi(\boldsymbol{x}) ( \boldsymbol{\xi},  \nabla U(\boldsymbol{x}),
-\sqrt{h} \nabla U(\boldsymbol{x}) + \sqrt{2 \beta^{-1}} \boldsymbol{\xi}) \alpha_h(\boldsymbol{x}, \boldsymbol{X}_1^{\star}) \right)  \right) 
 \\
&  -  
\frac{1}{6} \left( 2 \beta^{-1} \E^{\boldsymbol{x}} \left( 
D^3 \Phi(\boldsymbol{x}) ( \boldsymbol{\xi}, \boldsymbol{\xi},
 \nabla U(\boldsymbol{x}) ) \alpha_h(\boldsymbol{x}, \boldsymbol{X}_1^{\star}) \right)  \right) 
 \\
& +  \int_0^1 (1-t) \E^{\boldsymbol{x}} \left( \mathcal{L}^2 \Phi (\boldsymbol{Y}(h t)) \right)  dt  
 \\
R_2 = &
{1\over 4!} \int_0^1(1-t)^4 
\E^{\boldsymbol{x}} \left( 
D^4 \Phi(\boldsymbol{X}(t)) (\boldsymbol{X}_1^{\star} - \boldsymbol{x} )^4 \alpha_h(\boldsymbol{x}, \boldsymbol{X}_1^{\star}) 
\right)
\;.
\end{equs}
We now bound each of these terms separately.   The estimates that follow will often rely on
the hypothesis that $U(\boldsymbol{x}) < 1/\sqrt{h}$ together with Assumption~\ref{sa}~(D) 
which implies that  the $\ell$th derivative of $\Phi$ satisfies: 
\begin{equation} \label{DellPhibnd}
\| D^{\ell} \Phi (\boldsymbol{x}) \| \le C U^{\ell}(\boldsymbol{x}) \Phi(\boldsymbol{x})\;,
\end{equation}
for $\ell=1,2,3, 4$.

Since the term $I_{1/2}$ in \eqref{MALAphierror} involves an odd function of $\boldsymbol{\xi}$, 
one can rewrite it as  \[
\frac{I_{1/2}}{ \sqrt{2 \beta^{-1}} }= 
 \E^{\boldsymbol{x}} \left( \alpha_h(\boldsymbol{x}, \boldsymbol{X}_1^{\star}) 
  D \Phi(\boldsymbol{x}) \boldsymbol{\xi}  \right) = 
  \E^{\boldsymbol{x}} \left( (\alpha_h(\boldsymbol{x}, \boldsymbol{X}_1^{\star}) - 1)
D \Phi(\boldsymbol{x}) \boldsymbol{\xi}  \right) \;.
\]
Using \eqref{DellPhibnd}, we infer that 
\begin{equ}
| I_{1/2}|  \le  C U(\boldsymbol{x}) \Phi(\boldsymbol{x})  
 \E^{\boldsymbol{x}} \left( | 1- \alpha_h(\boldsymbol{x}, \boldsymbol{X}_1^{\star}) |^2 \right)^{1/2}   
  \le \tilde{C} U^3(\boldsymbol{x}) \Phi(\boldsymbol{x})   h^{3/2}   \;,
\end{equ}
where we used Lemma~\ref{MALAstagnationprobability} in the last inequality.  One can similarly bound $I_{3/2}$ 
since it also involves an odd function of $\boldsymbol{\xi}$.  The term $I_1$ is of
the form where Lemma~\ref{MALAstagnationprobability} can be directly applied
after using the Cauchy-Schwarz inequality and Assumption~\ref{sa}~(D).
The  terms in $I_2$ without integrals are bounded in a similar fashion, but 
without the need to invoke Lemma~\ref{MALAstagnationprobability}.

Note now that 
\[
| \mathcal{L}^2 \Phi (\boldsymbol{y}) | \le C U^{4}(\boldsymbol{y}) \Phi(\boldsymbol{y}) \;,
\]
which is a Lyapunov function for the true solution by Remark~\ref{SDEdriftcondition2},
so that the integrand 
appearing in $I_2$ is bounded by: 
\begin{equation}  \label{solnremainderbnd}
| \E^{\boldsymbol{x}} \left( \mathcal{L}^2 \Phi (\boldsymbol{Y}(r)) \right) | \le
C \E^{\boldsymbol{x}} \left( U^4(\boldsymbol{Y}(r)) \Phi(\boldsymbol{Y}(r)) \right) \le 
C U^4(\boldsymbol{x}) \Phi(\boldsymbol{x}) \;.
\end{equation}

Finally, we describe how to bound $R_2$ in \eqref{MALAphierror}.
It follows from  \eqref{DellPhibnd}  that \begin{align*}
R_2 &= \frac{1}{4!}  \left| \int_0^1(1-t)^4 
\E^{\boldsymbol{x}} \left(  D^4 \Phi(\boldsymbol{X}(t)) 
(\boldsymbol{X}_1^{\star} - \boldsymbol{x} )^4 \alpha_h(\boldsymbol{x}, \boldsymbol{X}_1^{\star}) \right) \right| \\
& \le 
C \E^{\boldsymbol{x}} \left( 
 U^4(\boldsymbol{X}(s)) \Phi(\boldsymbol{X}(s)) |\boldsymbol{X}_1^{\star} - \boldsymbol{x}|^4
 \alpha_h(\boldsymbol{x}, \boldsymbol{X}_1^{\star})
 \right) \;,
\end{align*}
so that our claim follows if we can show that the bound
\begin{equ}
\label{quarticintegrandbnd}
\E
 U^{4}(\boldsymbol{X}(s)) \Phi(\boldsymbol{X}(s)) |\boldsymbol{X}^{\star}( \boldsymbol{\xi}) - \boldsymbol{x}|^4
 \alpha_h(\boldsymbol{x}, \boldsymbol{X}^{\star}( \boldsymbol{\xi}) )  \le C U^{4}(\boldsymbol{x}) \Phi(\boldsymbol{x}) h^2 \;,
\end{equ}
holds uniformly for $s \in [0,1]$,
where $\boldsymbol \xi$ is a normally distributed random variable.
Here, we have introduced the shorthand notation
\begin{equ}
\boldsymbol{X}^\star(\boldsymbol{\xi}) = \boldsymbol{x} - h \nabla U(\boldsymbol{x}) + \sqrt{2 h \beta^{-1}} \boldsymbol{\xi} \;.
\end{equ}
Note that for all $\boldsymbol{x}$ satisfying $ U(\boldsymbol{x} ) < 1/ \sqrt{h}$, we have the bound
\begin{equ} \label{hyponx}
| \boldsymbol{X}^\star(\boldsymbol{\xi})  - \boldsymbol{x}| \le C\sqrt h (1+|\boldsymbol{\xi}|)\;.
\end{equ}

Hence, to prove \eqref{quarticintegrandbnd} it suffices to show that\[
\E \left( (1+ | \boldsymbol{\xi} |^4) 
 U^{4}(\boldsymbol{X}(s)) \Phi(\boldsymbol{X}(s)) \alpha_h(\boldsymbol{x}, \boldsymbol{X}^\star(\boldsymbol{\xi}))
 \right)  \le C U^{4}(\boldsymbol{x}) \Phi(\boldsymbol{x}) \;.
\]
We can then use the Cauchy-Schwarz inequality to get rid of the factor $(1+ | \boldsymbol{\xi} |^4)$,
so that it suffices to show that 
\begin{equ}[e:boundWanted]
\E \left(F_\theta (U(\boldsymbol{X}(s))) \alpha_h(\boldsymbol{x}, \boldsymbol{X}^\star(\boldsymbol{\xi}))
 \right)  \le C F_\theta(U(\x)) \;,
\end{equ}
where we defined the shorthand notation
\begin{equ}
F_\theta(u) = u^8 e^{2\theta u}\;.
\end{equ}
Our next step is to turn the occurrences of $\boldsymbol{X}(s)$ in this expression into $\boldsymbol{X}^\star(\boldsymbol{\xi})$.
In order to do this, we use the fact that Assumption~\ref{sa}~(C) implies that $U$ is `almost' convex.
Indeed, choose any $x,y \in \R^n$ and set $x_s = (1-s)x + sy$, so that one has the identity
\begin{equ}
U(x_s) = (1-s)U(x) + sU(y) + s(1-s) \int_0^1 \scal{\nabla U(x_{st})- \nabla U(x_{st + 1-t}), y-x}\,dt\;.
\end{equ}
Since $x_{st + 1-t} - x_{st} = (1-t)(y-x)$, it then follows from Assumption~\ref{sa}~(C) that
\begin{equs}
U(x_s) &\le (1-s)U(x) + sU(y) + C s (1-s) |x-y|^2 \int_0^1 (1-t) \,dt \\
&\le (1-s)U(x) + sU(y) + C |x-y|^2\;,\label{e:almostconvex}
\end{equs}
for some constant $C$ independent of $s \in [0,1]$. Note also that there exists a constant $C$ such that
the bound
\begin{equ}[e:boundFuv]
F_\theta(u+v) \le C F_\theta(u) \exp(Cv)\;,
\end{equ}
holds uniformly for all $u,v$ such that $u \ge 1$ and $v \ge 0$.

Since furthermore $F_\theta$ is convex, we deduce from \eref{e:boundFuv} and \eref{e:almostconvex}
that
\begin{equ}
F_\theta (U(\boldsymbol{X}(s))) \le 
C \exp\bigl(C|\boldsymbol{X}^\star(\boldsymbol{\xi})-\x|^2\bigr) \bigl((1-s)F_\theta(U(\x)) 
+ sF_\theta(U(\boldsymbol{X}^\star(\boldsymbol{\xi})))\bigr)\;.
\end{equ}
To bound the first term in this expression, note that it follows from \eqref{hyponx} that 
\begin{equ}
  \E  \exp\left( C | \boldsymbol{X}^\star(\boldsymbol{\xi})-\boldsymbol{x} |^2 \right)
 \le \E \exp\left( C h (1 + | \boldsymbol{\xi} |^2 )   \right) \le C\;,
\end{equ} 
 so that it is bounded by some multiple of $F_\theta(U(\x))$.

Combining this bound with \eref{e:boundWanted} and the Cauchy-Schwarz inequality, we
conclude that it remains to show that 
\begin{equ}
\E \left(F_\theta^2 (U(\boldsymbol{X}^\star(\boldsymbol{\xi}))) \alpha_h(\boldsymbol{x}, \boldsymbol{X}^\star(\boldsymbol{\xi}))
 \right)  \le C F_\theta^2(U(\boldsymbol{x})) \;.
\end{equ}
Since $\alpha_h$ is bounded, we can reduce ourselves to showing that
\begin{equ}[e:realwanted]
\E \left(F_\theta^2 (U(\boldsymbol{X}^\star(\boldsymbol{\xi}))) \alpha_h(\boldsymbol{x}, \boldsymbol{X}^\star(\boldsymbol{\xi}))\,,\;
U(\boldsymbol{X}^\star(\boldsymbol{\xi})) \ge U(\x) \right)  \le C F_\theta^2(U(\boldsymbol{x})) \;.
\end{equ}
Note now that one has from the definition \eref{MALAacceptreject} of $\alpha_h$ the bound
\begin{equ}
\alpha_h(\x,\y) \le {q_h(\y,\x) \over q_h(\x,\y)}\exp \bigl(\beta U(\x) - \beta U(\y)\bigr)\;,
\end{equ}
where $q_h$ denotes the one-step transition probabilities for forward Euler. 
The left-hand side of \eref{e:realwanted} can therefore be bounded by
\begin{equ}
\int_{U( \y) \ge U(\x)} F_\theta^2(U(\y)) q_h(\y,\x) \exp \bigl(\beta U(\x) - \beta U( \y)\bigr)\,d \y\;.
\end{equ}
We break this integral into two regions by setting
\begin{equ}
\CR_1 = \{ \y\,:\, U(\x) \le U(\y) \le \alpha h^{-1/2}\}\;,\quad 
\CR_2 = \{ \y\,:\, U(\y) \ge \alpha h^{-1/2}\}\;,
\end{equ}
for some $\alpha > 0$ to be determined.

Observe now that for $\y \in \CR_1$, one has the bound
\begin{equs}
q_h(\y, \x) &= 
(4 \pi \beta^{-1} h)^{-n/2} \exp\left( - \frac{\beta}{4h} | \x - \y + h \nabla U(\y) |^2 \right) \\
&\le
(4 \pi \beta^{-1} h)^{-n/2} \exp\left( - \frac{\beta}{8h} | \x- \y | ^2 + \frac{\beta h}{4}  | \nabla U( \y) |^2 \right) \\
&\le
C h^{-n/2} \exp\left( - \frac{\beta}{8h} | \x - \y | ^2 \right) \;,\label{e:boundqhfinal}
\end{equs}
where $C$ depends on the choice of $\alpha$, but not on $h$.
Furthermore, we have the bound
\begin{equs}
F_\theta^2(U(\y)) &\exp \bigl(\beta U(\x) - \beta U(\y)\bigr)\,d \y \label{e:boundFT} \\
& \le F_\theta^2(U(\x)) \exp\bigl(C |\x- \y| + (\beta - 2\theta) \bigl(U(\x) - U(\y)\bigr)\bigr)\;,
\end{equs}
where we have used Lemma~\ref{lem:Uexp} in order to obtain the last inequality.
Combining \eref{e:boundFT} and \eref{e:boundqhfinal} and using the fact that $U(\y) \ge U(\x)$ on $\CR_1$, 
we obtain indeed the bound
\begin{equ}
\int_{\CR_1} F_\theta^2(U( \y)) q_h(\y,\x) \exp \bigl(\beta U(\x) - \beta U(\y)\bigr)\,dy \le C F_\theta^2(U(\x))\;.
\end{equ}
Finally, in order to bound the integral over $\CR_2$, we make use of the fact that 
$q_h(\y,\x) \le C h^{-n/2}$, so that , combining this with \eref{e:boundFT}, we have the bound
\begin{equs}
\int_{\CR_2} &F_\theta^2(U(\y)) q_h(\y,\x) \exp \bigl(\beta U(\x) - \beta U(\y)\bigr)\,d \y \\
&\le C h^{-n/2} F_\theta^2(\x) \int_{\CR_2} \exp\bigl(C |\x- \y| + (\beta - 2\theta) \bigl(U(\x) - U( \y)\bigr)\bigr)\,d \y\\
&\le C h^{-n/2} F_\theta^2(\x) \exp(\beta h^{-1/2}) \int_{\CR_2} \exp\bigl(- \delta U( \y)\bigr)\,d \y\;,
\end{equs}
for some fixed constant $\delta > 0$. Here, we have made use of the fact that $U( \x) \le h^{-1/2}$ by
assumption, and that $U$ grows faster than quadratically by Assumption~\ref{sa}~(A).
It follows from \eref{e:boundHigh} and the definition of $\CR_2$ that 
\begin{equ}
\int_{\CR_2} \exp\bigl(- \delta U( \y)\bigr)\,dy\le \exp\Bigl(-{\beta \delta \alpha\over 2} h^{-1/2}\Bigr)\;,
\end{equ}
so that the requested bound follows, provided that we choose $\alpha$ sufficiently large so that $\alpha > 2\delta^{-1}$.
\end{proof}

The following lemma is useful to bound the average rejection probability of MALA.

%%%%%%%%%%%%%%%%%%%%%%%%%%%%%%%%%%%%%%%%%%%%%

\begin{lemma}[see also \cite{BoVa2010A}]  \label{MALAstagnationprobability}
For every $p \in \mathbb{N}$, there exists an $h_c>0$ and a constant $C >0$, such 
that for any $h<h_c$ the bound \[
 \E^{\boldsymbol{x}} \Bigl(  \left| 1 - \alpha_h(\boldsymbol{x},\boldsymbol{X}^{\star}_1) \right|^p \Bigr) 
\le C U^{2 p} (\boldsymbol{x}) h^{3 p/2} 
\]
 holds for all $\boldsymbol{x} \in \R^n$  satisfying $U(\boldsymbol{x}) < h^{-1/2}$. 
\end{lemma}

\begin{proof}
Introduce the function $G: \mathbb{R}^n \times \mathbb{R}^n \to \mathbb{R}$ given by
\begin{align*}
G(\boldsymbol{x}, \boldsymbol{y}) &= 
   U(\boldsymbol{y}) - U(\boldsymbol{x}) 
 - \frac{1}{2} \left\langle \nabla U(\boldsymbol{y}) + \nabla U(\boldsymbol{x}), \boldsymbol{y}-\boldsymbol{x} \right\rangle \\
+& \frac{h}{4} \left( \left| \nabla U(\boldsymbol{y}) \right|^2 - \left| \nabla U(\boldsymbol{x}) \right|^2 \right)  \;,
\end{align*}
and the set
\[
R(\boldsymbol{x}) = \{ \boldsymbol{y}  \in \mathbb{R}^n~~|~~ G(\boldsymbol{x}, \boldsymbol{y}) > 1 \} \;.
\]
By \eqref{MALAacceptreject} it follows that, for $\boldsymbol \xi$ a normally distributed random variable, one has
\begin{equ}
 \E^{\boldsymbol{x}} \left| 1 - \alpha_h(\boldsymbol{x},\boldsymbol{X}^{\star}_1) \right|^p 
  = \E
\left|   1 -  \bigl(1 \wedge \exp(-\beta G(\boldsymbol{x}, \boldsymbol{X}^\star(\boldsymbol{\xi}) ) \bigr)
\right|^p \;,
\end{equ}
where we have used the shorthand notation
\begin{equ}
\boldsymbol{X}^\star(\boldsymbol{\xi}) = \boldsymbol{x} - h \nabla U(\boldsymbol{x}) + \sqrt{2 h \beta^{-1}} \boldsymbol{\xi} \;.
\end{equ}
Since $| 1 - (1\wedge e^{-x})| \le |x|$ for every $x\in \R$, it follows that
\begin{equ}
 \E^{\boldsymbol{x}} \left| 1 - \alpha_h(\boldsymbol{x},\boldsymbol{X}^{\star}_1) \right|^p 
  \le \E
\left| \beta G(\boldsymbol{x}, \boldsymbol{X}^\star(\boldsymbol{\xi}) )
\right|^p \;.
\end{equ}
Introduce the interpolant \[
\boldsymbol{X}(t) = \boldsymbol{x} - t ( h \nabla U(\boldsymbol{x}) - \sqrt{2 h \beta^{-1}} \boldsymbol{\xi} )\;,
\]
so that $\boldsymbol{X}(0) = \boldsymbol{x}$ and
$\boldsymbol{X}(1) = \boldsymbol{X}^{\star}(\boldsymbol{\xi})$.  An straightforward but 
tedious calculation yields the identity
\begin{equs}
G(\boldsymbol{x}, \boldsymbol{X}^{\star}(\boldsymbol{\xi}) ) &=  
\frac{h}{2}  \int_0^1  D^2 U(\boldsymbol{X}(t)) \bigl(\nabla U(\boldsymbol{X}(t)), \boldsymbol{X}^{\star}(\boldsymbol{\xi})- \boldsymbol{x} \bigr) dt \\
& \qquad + \frac{1}{2}  \int_0^1  t(t-1)\, D^3 U( \boldsymbol{X}(t)) (\boldsymbol{X}^{\star}(\boldsymbol{\xi})- \boldsymbol{x})^3\,dt  \;.
\end{equs}
(Here we interpret $D^3 U(x) y^3$ as being the trilinear form $D^3 U(x)$ applied to the triple $(y,y,y)$.)
Note now that for all $\boldsymbol{x}$ satisfying $ U(\boldsymbol{x} ) < 1/ \sqrt{h}$, we have the bound
\begin{equ}
|\boldsymbol{X}^{\star}(\boldsymbol{\xi})- \boldsymbol{x}| \le C\sqrt h (1+|\boldsymbol{\xi}|)\;.
\end{equ}
On the other hand, we know from Assumption~\ref{sa}~(D) that $\|D^{(k)}U(x)\| \le C U(x)$ for $k = 1,2,3$ and
it follows from Lemma~\ref{lem:Uexp} that \[
U(\boldsymbol{X}(r))
\le
\exp\left(C\sqrt{h} (1 + | \boldsymbol{\xi} |)\right) U(\boldsymbol{x})\;.  
\]
for all $r \in [0, 1]$.  Combining these bounds, we obtain
\begin{equ}
|G(\boldsymbol{x}, \boldsymbol{X}^{\star}(\boldsymbol{\xi}) ) | \le C h^{3\over 2} U^{2}( \boldsymbol{x}) (1+| \boldsymbol{\xi} |)^3
\exp\left(C\sqrt{h} (1 + | \boldsymbol{\xi} |)\right) \;,
\end{equ}
for some constant $C>0$. Since the expression involving $\boldsymbol{\xi}$ has moments of all orders that are independent of $h$,
the result follows.
\end{proof}

%%%%%%%%%%%%%%%%%%%%%%%%%%%%%%%%%%%%%%%%%%%%%

In the following lemmas we prove a local drift condition for forward Euler.  
As mentioned the `strong' Lyapunov function $\Phi(\boldsymbol{y})$ is not integrable 
with respect to the transition probability of forward Euler since its tails are lighter than 
Gaussian.   But $U(\boldsymbol{y})^{\ell}$ is integrable as a consequence of 
Assumption~\ref{sa}~(D) which ensures that it grows at most exponentially fast.  We will show that single-step accuracy of forward 
Euler implies that it locally inherits this weaker Lyapunov function.

\begin{lemma} \label{Eulerlocaldrift}
Let $\tilde{\boldsymbol{X}}_1$ denote forward Euler after one step.  Then there exists a 
constant $C_{\ell} >0$ such that for every $E>0$ and $\ell \in \mathbb{N}$ the bound
\[
 \E^{\boldsymbol{x}} \bigl( U^{\ell}(\tilde{\boldsymbol{X}}_1) \bigr)  \le 
 ( e^{-\gamma_{\ell} h} + C_{\ell} U(\boldsymbol{x})^2 h^2) U^{\ell} (\boldsymbol{x})
 + \frac{K_{\ell}}{\gamma_{\ell}} (1- e^{-\gamma_{\ell} h})
 \]
 holds for all $\boldsymbol{x} \in \R^n$  satisfying $U(\boldsymbol{x}) < h^{-1/2}$.
\end{lemma}

\begin{proof}
Since
\[
U^{\ell}(\boldsymbol{x}) \le  e^{\ell C |\boldsymbol{x}|} U^{\ell}(\boldsymbol{0})
\]
by Lemma~\ref{lem:Uexp}, $U^{\ell}(\boldsymbol{x})$ is integrable with respect to 
Gaussian measures for every $\ell \in \mathbb{N}$.   
Thus, $(\tilde{P}_h U^{\ell})(\boldsymbol{x})$ is finite
(recall, $\tilde{P}_h$ is the transition probability for forward Euler).

Denoting by $\boldsymbol{Y}(h)$ the true solution after time $h$, we write
\[
 \E^{\boldsymbol{x}} \bigl( U^{\ell}(\tilde{\boldsymbol{X}}_1) \bigr) =
  \E^{\boldsymbol{x}} \bigl( U^{\ell}(\boldsymbol{Y}(h)) \bigr) +
  \E^{\boldsymbol{x}} \bigl( U^{\ell}(\tilde{\boldsymbol{X}}_1) - U^{\ell}(\boldsymbol{Y}(h)) \bigr) \;.
\]
Remark~\ref{SDEdriftcondition2} then implies that there are positive constants  $\gamma_\ell$ and $K_\ell$ such that
\[
 \E^{\boldsymbol{x}} \bigl( U^{\ell}(\tilde{\boldsymbol{X}}_1) \bigr) \le
  e^{-\gamma_{\ell} h} U^{\ell}(\boldsymbol{x}) + \frac{K_{\ell}}{\gamma_{\ell}} (1-e^{-\gamma_{\ell} h}) +
 \bigl| \E^{\boldsymbol{x}} \bigl( U^{\ell}(\tilde{\boldsymbol{X}}_1) - U^{\ell}(\boldsymbol{Y}(h)) \bigr) \bigr | \;,
\]
and the approximation result between forward Euler and the true solution
given in Lemma~\ref{EulerUellerror} below implies the desired result.
\end{proof}

An immediate corollary of this bound is given by
\begin{corollary}\label{globalEuler}
For every $\ell > 1$ there exist positive constants $\alpha_\ell$ and $K_\ell$ such that the bound
\begin{equ}
 \E^{\boldsymbol{x}} \bigl( U(\tilde{\boldsymbol{X}}_1) \wedge \alpha_\ell h^{-1/2}\bigr)^{\ell}
 \le \bigl( U(\x) \wedge \alpha_\ell h^{-1/2}\bigr)^{\ell} + h\,K_\ell\;,
\end{equ}
holds for every $\x \in \R^n$.
\end{corollary}

\begin{proof}
It follows from Lemma~\ref{Eulerlocaldrift} that, provided that $\alpha_\ell$ is small enough,
one has 
\begin{equ}
 \E^{\boldsymbol{x}} \bigl( U^{\ell}(\tilde{\boldsymbol{X}}_1) \wedge \alpha_\ell h^{-1/2}\bigr)
\le  \E^{\boldsymbol{x}} \bigl( U^{\ell}(\tilde{\boldsymbol{X}}_1) \bigr)
 \le U^{\ell}(\x) + h\,K_\ell\;,
\end{equ}
for all $\x$ such that $U(\x) \le \alpha_\ell h^{-1/2}$. On the other hand, one has the obvious
bound 
\begin{equ}
 \E^{\boldsymbol{x}} \bigl( U(\tilde{\boldsymbol{X}}_1) \wedge \alpha_\ell h^{-1/2}\bigr)^\ell
 \le \bigl(\alpha_\ell h^{-1/2}\bigr)^\ell\;,
\end{equ}
which is valid for all $\x$. Collecting both bounds concludes the proof.
\end{proof}

%%%%%%%%%%%%%%%%%%%%%%%%%%%%%%%%%%%%%%%%%%%%%

\begin{lemma} \label{EulerUellerror}
Let $\tilde{\boldsymbol{X}}_1$ and $\boldsymbol{Y}(h)$ denote forward Euler and the true
solution after one step, respectively.  For every $\ell \in \mathbb{N}$, there exists a constant 
$C_{\ell}>0$ such that the bound \[
\left| \E^{\boldsymbol{x}} \left( U^{\ell} (\tilde{\boldsymbol{X}}_1)- U^{\ell} (\boldsymbol{Y}(h)) \right) \right|
 \le C_{\ell} h^2 U^{\ell + 2}(\boldsymbol{x})
 \]
 holds for all $\boldsymbol{x} \in \R^n$  satisfying $U(\boldsymbol{x}) < h^{-1/2}$ and 
 for all $h<1$.
\end{lemma}

\begin{proof}
Observe that a single step of forward Euler is equivalent in law to the
following Langevin diffusion with constant drift:  \[
\tilde{\boldsymbol{X}}_1 \sim \boldsymbol{X}(h)  
\]
where the process $\boldsymbol{X}$ satisfies
\begin{equation}
d \boldsymbol{X} = 
- \nabla U(\boldsymbol{x}) dt + \sqrt{2 \beta^{-1}} d \boldsymbol{W},
~~~\boldsymbol{X}(0) = \boldsymbol{x} \;.
\label{SDEstar}
\end{equation}
The infinitesimal generator of this process is given by: \[
(\mathcal{L}_h g)(\boldsymbol{y}) = 
- \nabla U(\boldsymbol{x})^T \nabla g(\boldsymbol{y}) 
+  \beta^{-1} \Delta g( \boldsymbol{y})  \;.
\]
Since $\mathcal{L}_h U^{\ell}(\boldsymbol{x}) = \mathcal{L}  U^{\ell}(\boldsymbol{x})$,  
an exact It\^o-Taylor expansion yields,
\begin{align*}
& \E^{\boldsymbol{x}} 
\Bigl( U^{\ell}(\tilde{\boldsymbol{X}}_1)  - U^{\ell}(\boldsymbol{Y}(h))  \Bigr)  =  \\
& \qquad \E^{\boldsymbol{x}} \Bigl( 
\int_0^h \int_0^s \left( 
\mathcal{L}_h^2 U^{\ell} (\boldsymbol{X}(r)) - \mathcal{L}^2 U^{\ell}(\boldsymbol{Y}(r))  
\right) dr ds  \Bigr)\;.
\end{align*}
The triangle inequality implies
\begin{equs}\label{Uellerrorexpansion}
\Bigl| \E^{\boldsymbol{x}} 
&\Bigl( U^{\ell}(\tilde{\boldsymbol{X}}_1)  - U^{\ell}(\boldsymbol{Y}(h))  \Bigr)  \Bigr| \le  \\
& \int_0^h \int_0^s  \Bigl|\E^{\boldsymbol{x}}  \Bigl( \mathcal{L}_h^2 U^{\ell} (\boldsymbol{X}(r)) \Bigr)\Bigr|\,dr\,ds    + 
\int_0^h \int_0^s \Bigl|\E^{\boldsymbol{x}} \Bigl(  \mathcal{L}^2 U^{\ell}(\boldsymbol{Y}(r))  \Bigr)\Bigr|\,dr\,ds\;.
\end{equs}
Assumption~\ref{sa}~(D) implies that there exists a positive constant $C$ such that \[
( \mathcal{L}^2 U^{\ell} ) (\boldsymbol{y})   \le C U^{\ell + 2}(\boldsymbol{y}) \;, ~~
 ( \mathcal{L}_h^2 U^{\ell} ) (\boldsymbol{y})  \le C U^2(\boldsymbol{x})   U^{\ell}(\boldsymbol{y}) \;,
\]
for all $\boldsymbol{x}, \boldsymbol{y} \in \mathbb{R}^n$.   
These inequalities bound the integrands in \eqref{Uellerrorexpansion}.  
For the second term, we have 
\begin{equ}[e:boundLUY]
 \left| \E^{\boldsymbol{x}} \Bigl( \mathcal{L}^2 U^{\ell}(\boldsymbol{Y}(r))   \Bigr)  \right|  \le 
C \E^{\boldsymbol{x}} \Bigl(  U^{\ell+2}(\boldsymbol{Y}(r))   \Bigr)   \le  \tilde{C} U^{\ell+2}(\boldsymbol{x}) \;,
\end{equ}
where Remark~\ref{SDEdriftcondition2} is used in the last step.

To bound the first term, note that
\begin{equ}[e:boundLU]
\left|  \E^{\boldsymbol{x}}  \Bigl( \mathcal{L}_h^2 U^{\ell} (\boldsymbol{X}(r)) \Bigr) \right|   \le 
C U^2(\boldsymbol{x}) \E^{\boldsymbol{x}} \Bigl(  U^{\ell} (\boldsymbol{X}(r))  \Bigr)   
\end{equ}
where $r \in [0, h]$. The definition of an Euler step yields
\[
 \E^{\boldsymbol{x}} \Bigl(  U^{\ell}(\boldsymbol{X}(r))   \Bigr)   =
 (2 \pi)^{-n/2} \int_{\mathbb{R}^n}  
 U^{\ell}(\boldsymbol{x} - r \nabla U(\boldsymbol{x}) + \sqrt{2 \beta^{-1} r} \boldsymbol{\xi})
 \exp\left(- \frac{ | \boldsymbol{\xi} |^2}{2} \right) d \boldsymbol{\xi}  \;.
\]
Since, by hypothesis, \[
r | \nabla U(\boldsymbol{x}) | \le h  | \nabla U(\boldsymbol{x}) | \le C h  | U(\boldsymbol{x}) | \le C \sqrt{h} \;,
\]
it follows from Lemma~\ref{lem:Uexp} that 
\[
U^{\ell}(\boldsymbol{x} -r \nabla U(\boldsymbol{x}) + \sqrt{2 \beta^{-1} r} \boldsymbol{\xi}) \le
\exp\left(\ell \sqrt{h} (C+\sqrt{2 \beta^{-1}} | \boldsymbol{\xi} |)\right) U^{\ell}(\boldsymbol{x})
\]
for all $r \in [0, h]$.  Therefore,  \begin{equation} \label{EulerUellBound}
\E^{\boldsymbol{x}} \bigl(  U^{\ell}(\boldsymbol{X}(r))   \bigr) \le C U^{\ell}(\boldsymbol{x}) \;,
\end{equation}
for some $C>0$ independent of $h$. Combining \eref{EulerUellBound}, \eref{e:boundLU} and \eref{e:boundLUY} 
and inserting these bounds into \eref{Uellerrorexpansion}
yields the required bound.
\end{proof}

\section{Conclusion}

In this paper we showed that MALA's lack of a spectral gap is not severe.   
In particular, our main result, Theorem~\ref{theo:main}, states its convergence to equilibrium happens at 
exponential rate up to terms exponentially small in time-stepsize.  This quantification 
relies on MALA exactly preserving the SDE's invariant measure and accurately 
representing the SDE's transition probability on finite time intervals.  The first
property is automatic since the target distribution in the Metropolis-Hastings step is
the SDE's equilibrium distribution.  Deriving the second property requires a 
generalization of finite-time estimates for MALA \cite{BoVa2010A} 
and forward Euler \cite{BaTa1995, MaStHi2002}.   This derivation involves obtaining
new results on the accuracy of MALA and forward Euler with respect to the true solution
of the SDE in the context where the drift is not be globally Lipschitz.

A key technical issue addressed in the proof of Theorem~\ref{theo:main} is that MALA locally inherits a 
Lyapunov function of the true solution $\Phi(x) = \exp( \theta U(x))$.  
Since $U$ grows faster than a quadratic function, the function $\Phi$ is not integrable with respect to 
a Gaussian measure including the transition probability of forward Euler.    Nevertheless, 
we prove integrability of $\Phi$ with respect to the transition probability of MALA as a 
consequence of MALA preserving an equilibrium measure whose tails decrease faster
than  $\Phi$ increases.

Finite-time accuracy implied MALA inherits a minorization and local drift condition 
from the SDE.   As a consequence the paper proved that its mixing time is nearby the 
mixing time of the SDE on compact sets.  The patching argument in 
Theorem~\ref{theo:main} compares MALA to a version of MALA with 
reflection on the boundary of these compact sets to boost this local property to a 
global property plus terms exponentially small in time-stepsize.

Finally, we note that the proof of Lemma~\ref{1term} motivates the
following question: is forward Euler a strongly or weakly convergent
method on finite time intervals?  The answer is no because a necessary
condition for a numerical method to converge on finite time intervals
is stability which we have shown forward Euler lacks for nonglobally
Lipschitz drifts.  However, the lemma does motivate using forward
Euler as a proposal chain in the Metropolis-Hastings algorithm to
sample from the equilibrium measure of the SDE.

%%%%%%%%%%%%%%%%%%%%%%%%%%%%%%%%%%%%%%%%%%%%%
%%%%%%%%%%%%%%%%%%%%%%%%%%%%%%%%%%%%%%%%%%%%%

\bibliography{nawaf}{}
\bibliographystyle{Martin}

 \medskip

\end{document}